\documentclass[a4paper,11pt,oneside]{article} 
\usepackage{amsfonts,amssymb}
\usepackage{color} \usepackage{mathrsfs} \let\mathcal\mathscr
\usepackage[all,ps,cmtip]{xy}

\usepackage[utf8]{inputenc}
\usepackage{fancyhdr}
\usepackage[T1]{fontenc}
\usepackage{graphicx}
\usepackage{tikz}
\usepackage[frenchb]{babel}
\usepackage[left=2cm,right=2cm,top=2cm,bottom=2cm]{geometry}
\usepackage{amsmath,amsfonts,amssymb}
\usepackage{eurosym}
\usepackage{skull}
\usepackage{mathenv}
\usepackage{pstricks}

\title{\sc Compact hyper-Kähler categories}
\author{\sc Roland Abuaf \footnote{E-mail :\textit{rabuaf@gmail.com}.}\\
             \sc With an appendix by Gr\'egoire Menet \footnote{Institudo de Matematica estatistica e Camputaçao Cientifica, UNICAMP, Cidade Universitaria Zeferio Vaz, Campinas, Brésil. E-mail :\textit{menet@ime.unicamp.br}.}}

\usepackage{amsfonts,amssymb}
\usepackage{color} \usepackage{mathrsfs} \let\mathcal\mathscr

\usepackage[all,ps,cmtip]{xy}

\DeclareGraphicsExtensions{eps}
\DeclareGraphicsExtensions{pdf}

\usepackage{amsmath}
\usepackage{latexsym}

\newtheorem{theo}{Theorem}[subsection]

\newtheorem{exem}[theo]{Example}
\newtheorem{rem}[theo]{Remark}
\newtheorem{prop}[theo]{Proposition}
\newtheorem{quest}[theo]{Question}

\newtheorem{defi}[theo]{Definition}

\newtheorem{lem}[theo]{Lemma}
\newtheorem{cor}[theo]{Corollary}
\newtheorem{conj}[theo]{Conjecture}

\def\DB{\mathrm{D^{b}}}
\def\OO{\mathcal{O}}
\def\w{\omega}

\def\DP{\mathrm{D^{perf}}}

\def\R0{\mathrm{R^{0}}}

\def\HH{\mathrm{HH}}
\def\Hh{\mathcal{H}om}
\def\HHH{\mathrm{Hom}}
\def\LL{\mathrm{\textbf{L}}}

\def\RR{\mathrm{\textbf{R}}}

\def\OO{\mathcal{O}}
\def\EE{\mathcal{E}}

\def\D{\mathcal{D}}
\def\d{\delta}

\def\M{\mathcal{M}}

\def\C{\mathcal{C}}
\def\F{\mathcal{F}}
\def\GG{\mathrm{G}}
\def\G{\mathcal{G}}
\def\H{\mathcal{H}}
\def\T{\mathcal{T}}
\def\A{\mathcal{A}}
\def\B{\mathcal{B}}

\def\X{\mathcal{X}}
\def\w{\omega}

\DeclareMathOperator{\sing}{Sing}

\DeclareMathOperator{\Pic}{Pic}

\DeclareMathOperator{\CC}{\mathbb{C}}
\DeclareMathOperator{\ch}{ch}

\newcommand{\eq}[1][r]
{\ar@<-3pt>@{-}[#1]
\ar@<-1pt>@{}[#1]|<{}="gauche"
\ar@<+0pt>@{}[#1]|-{}="milieu"
\ar@<+1pt>@{}[#1]|>{}="droite"
\ar@/^2pt/@{-}"gauche";"milieu"
\ar@/_2pt/@{-}"milieu";"droite"}

\newcommand{\incl}[1][r]
  {\ar@<-0.2pc>@{^(-}[#1] \ar@<+0.2pc>@{-}[#1]}

\newenvironment{proof}
{
\noindent
\textit{\underline{Proof}} :\\
$\blacktriangleright\;$%
}
{\hspace{\stretch{1}}%
$\blacktriangleleft$}

\begin{document}

\maketitle

\begin{abstract}
We define and study the notion of \textit{hyper-Kähler category}. On the theoretical side, we focus on construction techniques and deformation theory of such categories. We also study in details some examples : \textit{non-commutative Hilbert schemes of points} on a $K3$ surface and a categorical resolution of a \textit{relative compactified Prymian} constructed by Markushevich and Tikhomirov.

\end{abstract}

\vspace{\stretch{1}}

\newpage

\tableofcontents

\begin{section}{Introduction}

\begin{subsection}{Background}
An irreducible holomorphically symplectic manifold (or compact hyper-K\"ahler manifold) is a smooth compact simply connected K\"ahler manifold with a unique (up to scalar) nowhere degenerate holomorphic $2$-form. Together with complex tori and Calabi-Yau manifolds, such varieties are building blocks for the decomposition of Ricci-flat manifolds (see \cite{bogomolov, beauville}). However, contrary to the former two, it is quite hard to produce many different examples (or to classify them all) of compact hyper-K\"ahler manifolds. Up to deformation, the following are the only known examples of such manifolds:
\begin{itemize}
\item  $S^{[n]}$, the Hilbert-Douady scheme of $n$ points on a $K3$ surface (see \cite{beauville}),
\item  $\mathrm{K}_n(A)$, the generalized K\"ummer variety of level $n$ associated to an abelian surface (see \cite{beauville}),
\item  A crepant resolution of $\M_S(2,0,4)$, the moduli space of semi-stable rank-$2$ torsion free sheaves with $c_1 = 0$ and $c_2 = 4$ on a $K3$ surface (see \cite{OG1}),
\item  A construction similar to the previous one, but where the $K3$ surface is replaced by an abelian surface (see \cite{OG2}).

\end{itemize} 

On the other hand, the derived categories of compact hyper-K\"ahler manifolds form an extremely interesting playground to test Kontsevich's Homological Mirror Symmetry conjecture. Indeed, one expects that such categories have a lot of autoequivalences which do not come from automorphisms of the complex structure but are instead related to Dehn twists along lagrangian projective spaces in the mirror manifold (see \cite{H-T}). \footnote{One also expects the mirror of a hyper-K\"ahler manifold to be a twistor deformation of itself (see \cite{verbi1, huy1}). Combined with Konstevich's HMS conjecture, this should give rather strong constraints on the derived category of a compact hyper-K\"ahler manifold.} Hence, it seems of high importance to have more examples of derived categories of compact hyper-K\"ahler manifolds. Or perhaps not so much derived categories of compact hyper-K\"ahler manifolds at such, but we definitively need more examples of triangulated categories which closely look like these derived categories.

\bigskip

The purpose of this paper is to introduce the notion of \textit{compact hyper-K\"ahler categories}, to study some of their basic properties and to provide some interesting new examples. Roughly speaking, a compact hyper-K\"ahler category of dimension $2n$ is a smooth compact simply connected category with a Serre functor given by the translation by $[2n]$ and endowed with a unique (modulo scalar) non-degenerate categorical $2$-form (all these notions will be made precise in the sections $2$ and $3$ of this paper).

One easily shows that if $X$ is an algebraic variety then $\DB(X)$ is a compact hyper-K\"ahler category if and only if $X$ is a compact hyper-K\"ahler manifold. Our main technique to construct new examples of such categories is based on Kuznetsov's theory of categorical crepant resolution of singularities (see \cite{kuz1}). It is well known that one can produce a lot of singular holomorphically symplectic varieties (see \cite{Mukmuk}) and that crepant resolutions of such varieties are holomorphically symplectic manifolds. Unfortunately, experience tells us that it is almost always impossible to find crepant resolutions of interesting singular holomorphically symplectic varieties (see \cite{Choy-Kiem, Choy-Kiem2, Kaledin-Lehn, Lehn-Sorger, KLS, Marku-Tikho, sacca}).

\bigskip

It is however not too difficult to produce \textit{categorical strongly crepant resolutions of singularities} of some nice singular holomorphically symplectic varieties. For instance, if $X$ is a hyper-K\"ahler manifold and $\GG$ is a finite group of automorphisms of $X$ preserving the symplectic form, then $X/ \GG$ admits a categorical strongly crepant resolution of singularities (see Theorem \ref{quotienttheo} below for a more general statement). The existence of categorical crepant resolution for all quotient singularities is in clear contrast with the known results in the commutative world. It is indeed a notoriously difficult problem to decide when a quotient singularity of dimension bigger than $4$ admits a crepant resolution \cite{Kaledin-quotient, BKR}.

\bigskip

Hence, it seems that examples of compact hyper-K\"ahler spaces are way easier to construct in the non-commutative setting. Given the scarcity of examples of commutative hyper-K\"ahler manifolds, the ease with which one can construct non-commutative incarnations of such varieties plainly justifies, in my opinion, the detailed study of such \textit{hyper-K\"ahler categories}. Furthermore, unexpected properties of such categories will probably be discovered in the near future and they will certainly shed a new light on the algebraic study of compact hyper-K\"ahler manifolds.

\end{subsection}

\begin{subsection}{Overview of the paper}

Let me give a quick overview of the theory of compact hyper-K\"ahler categories developed in this paper. First of all one would like to give a definition of compact hyper-K\"ahler categories which is invariant by equivalence. The work of Huybrechts and Niper-Wisskirchen \cite{Huy-Nieper} suggests that it is possible in some geometric cases. Indeed, they prove that if $X_1$ and $X_2$ are derived equivalent smooth projective varieties, then $X_1$ is hyper-K\"ahler if and only if $X_2$ is hyper-K\"ahler. A complete definition of compact hyper-K\"ahler categories will be given in section 3 of this paper. For now, let me focus on an important special case:

\begin{defi}
Let $X$ be a smooth projective variety and $\T \subset \DB(X)$ be a full admissible subcategory and assume furthermore that $\OO_{X} \in \T$. The category $\T$ is said to be \textbf{compact hyper-K\"ahler of dimension 2m (with respect to its embedding in $\DB(\X)$)} if the Serre functor of $\T$ is the shift by $2m$ and $H^{\bullet}(\OO_{X}) = \mathbb{C}[t]/t^{m+1}$, with $t$ homogeneous of degree $2$.
\end{defi}

This definition would be independent of the embedding if one could prove that for all smooth projective $Y$ such that $T$ is a full admissible subcategory of $\DB(Y)$ containing $\OO_Y$, one has $H^{\bullet}(\OO_Y) \simeq H^{\bullet}(\OO_X)$ (as graded algebras). This does not seem to be easy. Indeed, even in the case where $\T \simeq \DB(X) \simeq \DB(Y)$ and $X$ is hyper-K\"ahler, the proof given in \cite{Huy-Nieper} that $H^{\bullet}(\OO_Y) \simeq H^{\bullet}(\OO_X)$ relies on deep structural results for the Hochschild cohomology of compact hyper-K\"ahler manifolds. Nevertheless, one would expect that these two graded algebras are isomorphic whenever $X$ and $Y$ are derived equivalent. I discuss this invariance problem in more details in \cite{homounit}.

\bigskip

Given this definition, it is easy to check that $X$ is a hyper-K\"ahler manifold if and only if $\DB(X)$ is a compact hyper-K\"ahler category. Of course, one would like to construct new examples of such categories. In section $2$ (see Theorem \ref{hk-crepant}), we will prove:

\begin{theo} 
Let $Y$ be a projective manifold with Gorenstein rational singularities of dimension $2m$. Assume that $\omega_Y \simeq \OO_Y$ and that $H^{\bullet}(\OO_Y) \simeq \mathbb{C}[t]/(t^{m+1})$, with $t$ homogenous of degree $2$. Any categorical strongly crepant resolution of $Y$ is a compact hyper-K\"ahler category.
\end{theo}

This result provides us with a whole heap of examples of compact hyper-K\"ahler categories which are not deformation equivalent one to another. Indeed, let $X$ be a hyper-K\"ahler manifold and let $\GG_1$ and $\GG_2$ be two finite groups of automorphisms of $X$ preserving the symplectic form. Then, one proves (see Corollary \ref{finite}) that $\DB (Coh^{\GG_1}X)$ and $\DB(Coh^{\GG_2}X)$ are compact hyper-K\"ahler categories, where $Coh^{\GG}(X)$ is the category of $\GG$-equivariant sheaves on $X$. If $\GG_1$ and $\GG_2$ are chosen in a astute way, the orbifold topological Euler characteristic of $X/\GG_1$ is different from that of $X/\GG_2$. We deduce that $\DB(Coh^{\GG_1}X)$ and $\DB(Coh^{\GG_2}X)$ are not deformation equivalent.

\bigskip

Since the theory of compact hyper-K\"ahler spaces is based on examples which are often considered up to deformation equivalence, I would also like to discuss some properties of smooth deformations of triangulated categories of geometric origin. For example, it is often stated that Hochschild homology numbers of triangulated categories are invariant under deformation. I was unable to find a reference in the literature for such a general statement (see however \cite{Kaledin-cyclic} for some results on cyclic homology). Hence, I think it is worth providing a setting for which we can prove that the Hochschild homology numbers are indeed invariant under deformation. This will be done in the subsection $3.2.$ 

After introducing this context for deformation theory of triangulated categories of geometric origin and proving invariance of Hochschild homology for smooth deformations, I go on studying more precisely deformations of compact hyper-K\"ahler categories. The following proposition on ``small deformations'' and its proof are straightforward generalizations of the corresponding statement and proof for commutative compact hyper-K\"ahler manifolds: 

\begin{prop}
Let $X$ be a smooth projective variety, let $\T \subset \DB(X)$ be a full admissible subcategory and $B$ a smooth connected algebraic variety. Let $\D$ be a smooth deformation of $\T$ over $B$ with respect to $\pi : \X \longrightarrow B$. Assume that $\OO_{X} \in \T$ and that $\T$ is a compact hyper-K\"ahler category (with respect to its embedding in $\DB(X)$). Then, there exists a neighborhood $0 \in U \subset B$, such that $\D_b$ is compact hyper-K\"ahler (with respect to its embedding in $\DB(\X_b)$) for all $b \in U$.
\end{prop}

On the other hand, as far as ``long-time deformations" are concerned, one can show that a limit of (commutative) compact hyper-K\"ahler manifolds is again compact hyper-K\"ahler, provided it is projective. This result is not obvious and requires holonomy techniques to be proved. One would like to know if this holds true in the non-commutative world. I believe that such a deformation result would be important for the theory of compact hyper-K\"ahler categories and that it can not be proved without the design of new (probably powerful) geometrico-categorical techniques.

\begin{conj}[see Conjecture \ref{def-HK}]
Let $X$ be a smooth projective variety, let $\T \subset \DB(X)$ be a full admissible subcategory and $B$ a smooth algebraic variety. Let $\D$ be a deformation of $\T$ over $B$. Assume that $\D_b$ is compact hyper-K\"ahler for all $b \neq 0$. Then, the category $\D_0 = \T$ is compact hyper-K\"ahler.
\end{conj}

The first step toward such a result should establish that a deformation of a Calabi-Yau category is again a Calabi-Yau category. This already happens to be not obvious. Restricting to deformations of Calabi-Yau categories, one can demonstrate the following (see Proposition \ref{def-hk4}):

\begin{prop}
Let $X$ be a smooth projective variety, let $\T \subset \DB(X)$ be a full admissible subcategory which is Calabi-Yau of dimension $4$ and let $B$ a smooth algebraic variety. Let $\D$ be a smooth deformation of $\T$ over $B$ with respect to $\pi : \X \rightarrow B$. Assume that $\OO_{X} \in \T$ and that for all $b \neq 0$, the category $\D_b$ is compact hyper-K\"ahler of dimension $4$ (with respect to its embedding in $\DB(\X_b)$). Then, the category $\D_0 = \T$ is compact hyper-K\"ahler of dimension $4$ (with respect to its embedding in $\DB(\X_0)$).
\end{prop}

It would certainly be desirable to know if this result can be generalized to higher-dimensional cases. It would also be very interesting to discover which structural results known for the Hochschild cohomology rings of compact hyper-K\"ahler manifolds are still valid in the categorical context. For instance, if $X$ is a compact hyper-K\"ahler manifold of dimension $2m$ and $\HH^{\langle 2 \rangle}(X)$ is the sub-algebra of $\HH^*(X)$ generated by $\HH^2(X)$, Verbitsky \cite{verbi2} proved the following isomorphism:
\begin{equation*}
\HH^{\langle 2 \rangle}(X) \simeq S^*\HH^{2}(X)/\{a^{m+1}, \, \text{such that} \, q(a) = 0 \},
\end{equation*}
where $q$ is the Beauville-Bogomolov quadratic form. This result is heavily used in \cite{Huy-Nieper} to prove the derived invariance of the compact hyper-K\"ahler property for projective varieties. In my opinion, it would be fascinating to have a similar statement for the Hochschild cohomology of a compact hyper-K\"ahler category.

\bigskip

The last two sections of the paper are dedicated to the study of specific examples. In section $4$, we focus on a question that was asked to me by Misha Verbitesky. 
\begin{quest}
Let $S$ be a K3 surface, for which $\GG \subset \mathcal{S}_n$ does the quotient $S \times \cdots \times S/ \GG$ admits a non-commutative crepant resolution which is hyper-Kähler?
\end{quest}

 The answer to this question in the commutative world is due to Verbitski himself and he proves that only for $\GG = \mathcal{S}_n$ such a resolution exists (and the resolution is the Hilbert scheme of $n$ points on $S$). In the non-commutative world, there are quite a few more examples and we have the:

\begin{theo} 
Let $S$ be a K3 surface and $n \geq 1$. Let $\GG \subset \mathcal{S}_n$ acting on $S \times \cdots \times S$ by permutations. The quotient $S \times \cdots \times S/ \GG $ admits a categorical crepant resolution which is a hyper-Kähler category if and only if $\GG$ is one of the following:

\begin{itemize}
\item $\GG = \mathcal{S}_n$,

\item $\GG = \mathcal{A}_n$  (the alternating group),

\item $n=5$ and $\GG = \mathbb{F}_5^*$,

\item $n= 6$ and $\GG = \mathbb{P}GL_2(\mathbb{F}_5)$, 

\item $n=9$ and $\GG = \mathbb{P}GL_2(\mathbb{F}_8)$,

\item $n=9$ and $\GG = \mathbb{P}GL_2(\mathbb{F}_8) \rtimes \mathrm{Gal}(\mathbb{F}_8/\mathbb{F}_2)$.

\end{itemize}
Furthermore, this categorical resolution is always non-commutative in the sense of Van-den-Bergh.
\end{theo}

The study of the Betti cohomology ring of Hilbert schemes of points on a K3 surface is a classical subject where hyper-Kähler geometry, number theory and representation theory interact fruitfully with one another. We expect that the study of the Hochschild cohomology ring of the above hyper-Kähler categories should reveal interesting new connections between these three topics.

\bigskip

In the final section of this paper, we describe in some details a categorical strongly crepant resolution of the \textit{relative compactified Prymian} constructed by Markushevich and Tikhomirov in \cite{Marku-Tikho}. It is known that this fourfold is a singular irreducible symplectic variety of dimension $4$ which has no crepant resolution of singularities. Our main result in section $5$ is the:

\begin{theo} 

The relative compactified Prymian of Markushevich and Tikhomirov admits a categorical strongly crepant resolution which is a hyper-Kähler category of dimension $4$. The Hochschild cohomology numbers of this category are:

\begin{itemize}
\item $\mathrm{hh}^{0} = \mathrm{hh}^8 = 1$,
\item $\mathrm{hh}^{2} = \mathrm{hh}^6 = 16$,
\item $\mathrm{hh}^{4} = 206$.
\end{itemize}
\end{theo}

This result is based in an essential way on the computations of the Hodge numbers of some resolution of singularities of the Prymian. This is done in the appendix by Grégoire Menet. Using the deformation theory developed in section $3$, we prove the:

\begin{prop}
A small deformation of the categorical strongly crepant resolution of the relative compactified Prymian of Markushevich-Tikhomirov can never be equivalent to the derived category of a projective variety (and it thus provides a counter-example to conjecture $5.8$ in \cite{kuzCY}). 
\end{prop}

We expect in fact that any deformation of this category is never equivalent to the derived category of a projective variety. If our expectation is correct, then the moduli space of hyper-Kähler categories of dimension $4$ (if such an object exists) contains a connected component which is \textbf{purely non-commutative}! 
\end{subsection}

\bigskip

\textbf{Acknowledgements}. I am very thankful to Chris Brav, Victor Ginzburg, Daniel Huybrechts, Sasha Kuznetsov, Richard Thomas and Matt Young for very interesting discussions about the various possible definitions of compact hyper-K\"ahler categories and the properties one would expect them to enjoy. I would also like to thank my former PhD advisor Laurent Manivel for many helpful comments on some preliminary version of this work. I am especially grateful to Gregoire Menet for supplying me with the Hodge numbers I was looking for and to Misha Verbitsky for asking the question studied in section $4$ of this paper.

\end{section}

\begin{section}{Categorical crepant resolution of singularities}

As mentioned in the introduction, our examples of compact hyper-K\"ahler categories are based on the theory of categorical crepant resolutions of singularities. This notion has been developed in \cite{kuz1} and was further explored in \cite{theseabuaf,quotient}. 

\begin{subsection}{Definition and motivations}

Let us recall that a \textit{crepant} resolution of a normal Gorenstein algebraic variety $Y$ is a resolutions of singularities $\pi : X \rightarrow Y$ such that $\pi^* \w_Y = \w_{X}$, where $\w_Y$ is the dualizing line bundle of $Y$. Crepant resolutions are often considered to be minimal resolutions of singularities (see the first part of \cite{theseabuaf} for an extended discussion about minimality for resolutions of singularities). Une fortunately crepant resolutions of singularities are quite rare. The following example is very classical:

\begin{exem} \label{exe}

Let $Y$ be a cone over $v_2(\mathbb{P}^3) \subset \mathbb{P}(S^2 \mathbb{C}^4)$. The variety $Y$ is analytically equivalent to $\mathbb{C}^6/\{1,-1\}$. Hence, it is locally analytically $\mathbb{Q}$-factorial (see \cite{mori-kollar}, Chapter 5), so that it has no small resolution of singularities. Furthermore the blow-up of $Y$ along the vertex gives a resolution of singularities where the coefficient of the exceptional divisor in the dualizing bundle formula is $1$ (this is an obvious computation). As a consequence, the variety $Y$ has terminal singularities. Since it admits no small resolution, we find that $Y$ has no crepant resolution of singularities.

\end{exem}

Given a singularity which does not admit any crepant resolution, one still would like to know if it is possible to produce minimal resolutions from the point of view of category theory. Kuznetsov's insight is that such categorical "minimal" resolutions should be constructed as \textit{categorical crepant resolution} (see \cite{kuz1}, section $4$).

\begin{defi}
Let $Y$ be an algebraic variety and $\X$ be a smooth Deligne-Mumford stack. We say that $\X$ \textbf{homologically dominates} $Y$, if there exists a proper morphism $p : \X \rightarrow Y$, such that $\RR p_* \OO_{\X} \simeq \OO_Y$.
\end{defi}

Typical examples of such phenomenon include resolutions of singularities for a variety with rational singularities and the canonical projection from a smooth Deligne-Mumford stack to its coarse moduli space.
\begin{defi} \label{categorical-resolution}
Let $Y$ be an algebraic variety with Gorenstein rational singularities. Let $p : \X \rightarrow Y$ be a smooth Deligne-Mumford stack which homologically dominates $Y$. A \textbf{categorical resolution} of $Y$ is a full admissible subcategory $\T \subset \DB(\X)$ such that $\LL p^* \DP(Y) \subset \T$.
\end{defi}

In \cite{kuz1}, the definition of categorical resolution was restricted to the case where $X$ is a variety. A way more general notion of categorical resolution has been defined and studied by Kuznetsov and Lunts in \cite{kuz2}. The main advantage of their definition is that one can prove the existence of a categorical resolution for any scheme (!) of finite typer over $\mathbb{C}$. With Definition \ref{categorical-resolution}, we lie in the middle. The possibility to work with Deligne-Mumford stacks allows to produce interesting examples of non-commutative resolution of singularities (see \cite{quotient}). On the other hand, many elementary techniques and results from \cite{kuz1} are still valid when $X$ is a smooth Deligne-Mumfod stack, with proofs being exactly the same.

\begin{defi}[Categorical crepancy, \cite{kuz1}]
Let $Y$ be a an algebraic variety with Gorenstein rational singularities and $p : \X \rightarrow Y$ be a Deligne-Mumford stack homologically dominating $Y$. Let $\delta : \T \hookrightarrow \DB(\X)$ be a categorical resolution of $Y$ and let ${p_{\T}}_* : \T \rightarrow \DB(Y)$ be the composition of $\RR p_*$ with $\delta$. 
\begin{itemize}

\item We say that ${p_{\T}}_* : \T \rightarrow \DB(Y)$ is a \textbf{weakly crepant resolution} of $Y$, if for all $\F \in \DP(Y)$, we have:

\begin{equation*}
p_{\T}^* \F \simeq p_{\T}^{!} \F,
\end{equation*}
where $p_{\T}^*$ and $p_{\T}^{!}$ are the left and right adjoint to ${p_{\T}}_*$.

\item We say that ${p_{\T}}_* : \T \rightarrow \DB(Y)$ is a \textbf{strongly crepant resolution} of $Y$ if the following two conditions hold:

\begin{enumerate}
\item we have $\LL p^* \F \otimes_{\OO_{\X}} \d \G \in \T$, for all $\F \in \DP(Y)$ and $\G \in \T$, 
\item the identity functor is a relative Serre functor for $\T$ with respect to the map ${p_{\T}}_* $.
\end{enumerate}
\end{itemize} 
\end{defi}

Let us make a few comments on this definition. The first requirement in the definition of a categorical strongly crepant resolution is that $\T$ has a module structure over $\DP(Y)$ (see \cite{kuz1}, section $3$). Assume that $\T$ is a categorical strongly crepant resolution of a projective variety $Y$. Then the (absolute) Serre functor of $\T$ is given by the tensor product by $\pi^* \omega_Y [\dim Y]$. Note also that a categorical strongly crepant resolution of a Gorenstein rational singularity is automatically a categorical weakly crepant resolution of this singularity, but the converse is not true (see \cite{kuz1}, section 8). However, in the purely geometric setting (that is when $\T \simeq \DB(X)$, for some algebraic variety $X$), all these notion coincide \cite{theseabuaf}. I refer to \cite{kuz1, kuz4, abuaf1, spvdb} for the existence of categorical crepant resolution of determinantal varieties.

\end{subsection}

\begin{subsection}{Categorical crepant resolutions and quotient singularities}

In this sub-section, I will recall the main result of \cite{quotient}. It will be useful to construct new compact hyper-K\"ahler categories starting from a compact hyper-K\"ahler variety endowed with a finite group of symplectic automorphisms.

\begin{theo}[\cite{quotient}] \label{quotienttheo}
Let $X$ be a quasi-projective variety with normal Gorenstein quotient singularities and let $\X$ be a smooth separated Deligne-Mumford stack whose coarse moduli space is $X$. Assume that the dualizing line bundle of $\X$ is the pull back of the dualizing line bundle on $X$, then $\DB(\X)$ is a strongly crepant resolution of $X$.

Furthermore, there exists a sheaf of algebras $\A$ on $X$ such that $\DB(\X) \simeq \DB(X, \A)$. Hence, the pair $(X, \A)$ is a non-commutative crepant resolution of $X$ in the sense of Van den Bergh.

\end{theo}

Note that if $X$ is a normal quasi-projective variety with quotient singularities, there is always a smooth separated Deligne-Mumford stack associated to it as in the above statement (see Proposition $2.8$ of \cite{vistoli}). The non-trivial hypothesis (which can not be removed) is that the dualizing bundle of the Deligne-Mumford stack associated to $X$ is the pull back of the dualizing bundle on $X$. This amounts to check that on an \'etale atlas of $\X$, the line bundle $\omega_{\X}$ is equivariantly \footnote{for the isotropy groups of the fixed points of the \'etale atlas of $\X$} locally trivial. This finally boils down to checking that for any $x \in X$, there exists an \'etale neighborhood $U_x$ of $x \in X$, such that $U_x = V/ \GG$ where $V$ is a vector space and $\GG$ is a subgroup of $\mathrm{SL}(V)$. This holds in particular for a variety whose singularities are isolated points locally analytically equivalent to a cone over $v_2(\mathbb{P}^3) \subset \mathbb{P}(S^2 \mathbb{C}^4)$. 
\bigskip

In the local case, the above result was already known for a long time (see \cite{vdb}, for instance). The heart of Theorem \ref{quotienttheo} is the existence result of categorical crepant resolutions for quotient singularities in the \textbf{global setting}. Indeed, there is a priori no reason for the local resolutions constructed in $\cite{vdb}$ to glue globally. The main point of $\cite{quotient}$ is to exhibit a sheaf of non-commutative algebras which provides such a gluing of the local resolutions.

\end{subsection}
\end{section}

\begin{section}{Compact hyper-K\"ahler categories}

Recall that a holomorphically sympletic variety of dimension $2m$ is (in the projective case) a smooth projective variety $X$ having trivial canonical bundle and endowed with a $2$-form $\sigma \in H^0(X, \Omega^2_X)$, such that $\sigma^{\wedge m} \neq 0$. On says that $X$ is compact hyper-K\"ahler if $X$ is simply connected and $\sigma$ generates $H^0(X, \Omega^2_X)$. Since $\sigma$ defines an isomorphism $\sigma : \Omega_X \rightarrow T_X$, one can equivalently say that $X$ is holomorphically symplectic if $X$ is smooth simply connected projective with trivial canonical bundle and there exists a Poisson bracket $\theta \in H^0(X, \bigwedge^2 T_X)$, such that $\theta^{\wedge m} \neq 0$. Hence, one could be tempted to give the following definition:

\begin{defi}[naive definition of holomorphically symplectic categories] \label{naivedef}
Let $\T$ be a smooth compact triangulated category. We say that $\T$ is \textbf{holomorphically symplectic of dimension $2m$} if the shift by $2m$ is a Serre functor for $\T$ and there exists $\theta \in \HH^2(\T)$ such that $\theta^{\circ m} \neq 0$.
\end{defi}

Such a definition has the advantage to be invariant by equivalences. Its (non-negligible) drawback is that the derived categories of many non holomorphically symplectic varieties are then to be considered as holomorphically symplectic categories. Indeed, the Hochschild-Kostant-Rosenberg isomorphism \cite{markarian} shows that for $X$ smooth projective, there is a decomposition (compatible with products on both sides \cite{calaque-vdb, Huy-Nieper}):

\begin{equation*}
\HH^2(X) = H^0(X, \bigwedge^2 T_X) \oplus H^1(X,T_X) \oplus H^2(X, \OO_X).
\end{equation*}
Hence, with Definition \ref{naivedef}, the derived category of an abelian surface would be considered as a holomorphically symplectic category, which is something we want to avoid. The main problem here is to define categorically one of the algebras $H^0(X, \bigwedge^{\bullet} T_X)$ or $H^0(X, \bigwedge^{\bullet} \Omega_X)$ (the latter being isomorphic, by Hodge duality, to the algebra $H^{\bullet}(X,\OO_X)$). We will give such a definition in the first subsection below. 

\begin{subsection}{Homological units}

Let $X$ be a algebraic variety and let $\F \in \DB(X)$ be an object whose rank is not zero. Then the trace map:

\begin{equation*}
\mathrm{Tr} : \RR \Hh (\F,\F) \rightarrow \OO_X
\end{equation*}
splits and gives a splitting:

\begin{equation*}
 \HHH^{\bullet}(\F,\F) = H^{\bullet}(\OO_X) \oplus \HHH^{\bullet} (\F,\F)_0, 
\end{equation*}
where $\HHH^{\bullet} (\F,\F)_0$ is the graded vector space of trace-less endomorphisms. Hence, the algebra $H^{\bullet}(\OO_X)$ appears as a maximal direct factor of the endomorphisms algebra of any object in $\DB(X)$ which rank is not vanishing. We will see below that this algebra is an important categorical invariant.

\begin{defi} \label{homounit}
Let $\C$ be an abelian category with a non-trivial rank function and $\T$ be a full admissible subcategory in $\DB(\C)$. A graded algebra $\mathfrak{T}^{\bullet}$ is called a \textbf{homological unit} for $\T$ (with respect to $\C$), if $\mathfrak{T}^{\bullet}$ is maximal for the following properties : 
\begin{enumerate}

\item for any object $\F \in \T$, there exists a pair of morphisms $i_{\F} : \mathfrak{T}^{\bullet} \rightarrow  \HHH^{\bullet}(\F,\F)$ and $t_{\F} : \HHH^{\bullet}(\F, \F) \rightarrow \mathfrak{T}^{\bullet}$ with the properties:
 \begin{itemize}

 \item the morphism $i_{\F} : \mathfrak{T}^{\bullet} \rightarrow  \HHH^{\bullet}(\F,\F)$ is a graded algebra morphism which is functorial in the following sense. Let $\F, \G \in \T$ and let $a \in \mathfrak{T}^{k}$ for some $k$. Then, for any morphism $\psi : \F \rightarrow \G$, there is a commutative diagram:
\begin{equation*}
\xymatrix{ \F \ar[rr]^{i_{\F}} \ar[dd]^{\psi} & &\F [k] \ar[dd]^{\psi[k]} \\
& &  \\
\G \ar[rr]^{i_{\G}} & & \G [k]}
\end{equation*}

\item the morphism $t_{\F} : : \HHH^{\bullet}(\F, \F) \rightarrow \mathfrak{T}^{\bullet}$ is a graded vector spaces morphism which satisfies the dual functoriality property of $i_{\F} $.
\end{itemize}

\item for any $\F \in \T$ which rank (seen as an object in $\DB(\C)$) is not vanishing, the morphism $t_{\F}$ splits $i_{\F}$ as a morphism of graded vector spaces.

\end{enumerate}

With hypotheses as above, an object $\F \in \T$ is said to be \textup{unitary}, if $\HHH^{\bullet}(\F,\F) = \mathfrak{T}^{\bullet}$, where $\mathfrak{T}^{\bullet}$ is a homological unit for $\T$.

\end{defi}
Of course, one can not expect that all examples of homological units as defined above will be significant. In the main applications of the present paper, one will look at $\C = Coh(X), Coh^{\GG}(X)$ or $Coh(X, \alpha)$, where $X$ is a smooth projective variety, $\GG$ a reductive algebraic group acting linearly on $X$, $\alpha$ a Brauer class on $X$ and the rank function will be the obvious one. However, it is well possible that many new examples of homological units coming from representation theory will be discovered, so that it seems sensible to give a general definition that does not restrict to purely geometrical examples. 

Note also that the hypothesis of non-vanishing rank for the splitting is a technical hypothesis which is important. It would be very interesting to know if there are some non-trivial examples where the splitting occurs whatever the rank of the object.

\begin{exem}
\begin{enumerate}
\item Let $X$ be a smooth algebraic variety and $\alpha \in\mathrm{Br}(X)$, a class in the Brauer group of $X$. Consider $\C = Coh(X,\alpha)$, the category of coherent $\alpha$-twisted sheaves on $X$. One can define a rank function on $\C$ as being the rank of $\F$ when seen as an $\OO_X$-module. Then for any $\F \in \DB(\C)$, we have a trace map:

\begin{equation*}
\mathrm{Tr} : \RR \Hh_{\DB(\C)} (\F,\F) \rightarrow \OO_X
\end{equation*}

which splits when the rank of $\F$ is not zero. As a consequence, for all $\F \in \DP(\C)$, we have a graded algebra morphism:
\begin{equation*}
 H^{\bullet}(\OO_X) \rightarrow  \HHH^{\bullet}_{\DB(\C)}(\F,\F)
\end{equation*}
which is split (as a morphism of vector spaces) when the rank of $\F$ is not zero. The morphism $H^{\bullet}(\OO_X) \rightarrow  \HHH^{\bullet}_{\DB(\C)}(\F,\F)$ is given by $a \rightarrow id_{\F} \otimes a$, so that the functoriality property is clearly satisfied. Furthermore, if $L$ is a twisted line bundle in $\DB(Coh(X,\alpha))$, we have $\HHH^{\bullet}_{\DB(\C)}(L,L) = H^{\bullet}(\OO_X)$. Thus, $H^{\bullet}(\OO_X)$ is maximal for the properties required in Definition \ref{homounit} and it is a homological unit for $\C$.

\item Let $X$ be a smooth algebraic variety and $\GG$ be a reductive algebraic group acting linearly on $X$. For any $\F \in \DB(Coh^{\GG}(X))$, the trace map $\mathrm{Tr} : \RR \Hh (\F, \F) \rightarrow \OO_X$ is $\GG$-equivariant and it is split if the rank of $\F$ is non-zero. Hence, for all $\F \in \DB(Coh^{\GG}(X)) $, we have a graded algebra morphism:

\begin{equation*}
 H^{\bullet}(\OO_X)^{\GG} \rightarrow \HHH^{\bullet}_{\DB(Coh^{\GG}(X))}(\F, \F), 
\end{equation*}
which is split (as a morphism of vector spaces) when the rank of $\F$ is not zero. The morphism $H^{\bullet}(\OO_X)^{\GG} \rightarrow  \HHH^{\bullet}(\F,\F)$ is again given by $a \rightarrow id_{\F} \otimes a$, so that the functoriality property is also satisfied. Furthermore, if $L$ is a $\GG$-invariant line bundle on $X$, we have $\HHH^{\bullet}_{\DB(\C)}(L,L) = H^{\bullet}(\OO_X)^{\GG}$. Hence, the algebra $ H^{\bullet}(\OO_X)$ is maximal for the properties required in Definition \ref{homounit} and it is a homological unit for $\C$. This readily generalizes for any smooth Deligne-Mumford stack. Namely, if $\X$ is a smooth Deligne-Mumford stack, then $H^{\bullet}(\OO_{\X})$ is a homological unit for $\DB(\X)$. Note that all line bundles on $\X$ are unitary objects.
\end{enumerate} 
\end{exem}

One would like to know when the homological unit is unique and independent of the embedding in the derived category of an abelian category. This question seems to be interesting for itself and it it does not have an obvious answer. I discuss this invariance problem in \cite{homounit}), where I give some applications to the conjectural derived invariance of Hodge numbers. The following result is a slight generalization the fifth assertion of Theorem $2.0.10$ in \cite{homounit}:

\begin{theo} \label{homounitiso}
Let $X$ and $Y$ be smooth projective varieties of dimension $4$. Let $\Phi : \DB(X) \hookrightarrow \DB(Y)$ be a fully faithful functor such that $\OO_{Y} \in \Phi(\DB(X))$ and that $\mathbb{C}(y) \in \Phi(\DB(X))$ for generic $y \in Y$. Then we have an isomorphism of graded algebras:
\begin{equation*}
H^{\bullet}(\OO_X) \simeq H^{\bullet}(\OO_Y).
\end{equation*}
\end{theo}

Note that the assumptions $\OO_{Y} \in \Phi(\DB(X))$ and $\mathbb{C}(y) \in \Phi(\DB(X))$, for generic $y \in Y$, are not both superfluous. Indeed, if $X$ is a smooth subvariety of $Y$ and $\tilde{Y}$ is the blow-up of $Y$ along $X$, there is a fully faithful embedding $ \Phi : \DB(X) \hookrightarrow \DB(\tilde{Y})$. But we have $H^{\bullet}(\OO_X) \neq H^{\bullet}(\OO_{\tilde{Y}})$ in general. Of course, it is clear that $\OO_{\tilde{Y}} \notin \Phi(\DB(X))$ and that $\CC(\tilde{y}) \notin \Phi(\DB(X))$ for generic $\tilde{y} \in \tilde{Y}$.  

\bigskip

\begin{proof}
The proof follows closely the lines of the proof of Theorem $2.0.10$ in \cite{homounit}. First we will prove that there exists $L \in \Pic(X)$ such that the rank of $\Phi(L)$ is non-zero. We proceed by absurd. Assume that for all $L \in \Pic(X)$, the rank of $\Phi(L)$ is zero. First of all, using Orlov's representability Theorem, we can see $\Phi$ as a Fourier-Mukai kernel with kernel $\G \in \DB(X \times Y)$. This means that $\Phi(?) = \RR p_* \left( \LL q^*(?) \otimes \G \right)$, where $p$ and $q$ are the natural projections in the diagram:

\begin{equation*}
\xymatrix{ & &  \ar[lldd]_{q} X \times Y \ar[rrdd]^{p} & &  \\
& & & & \\
X & &  & & Y \\}
\end{equation*}

Since $Y$ is smooth, the vanishing of the rank of $\Phi(L)$ for all $L \in \Pic(X)$, implies the vanishing:

\begin{equation*}
\chi(\Phi(L) \otimes \CC(y)) = 0,
\end{equation*}
for generic $y \in Y$ and for all $L \in \Pic(X)$. This means that 
$$\chi(\RR p_* \left( \LL q^*(L) \otimes \G \right) \otimes \CC(y)) = 0,$$ 
for all $L \in \Pic(X)$ and generic $y \in Y$. Using the projection formula for $\RR p_*$ and the Leray spectral sequence for $\RR p_*$ and $\RR \Gamma$, we find that it is equivalent to:
$$\chi( \LL q^*(L) \otimes \G  \otimes \LL p^* \CC(y)) = 0,$$ 
for all $L \in \Pic(X)$ and generic $y \in Y$. As $\LL p^* \CC(y) = \OO_{X \times y}$, we find that:
$$\chi(L \otimes {j_y}^*\G) = 0,$$
 for all $L \in \Pic(X)$, generic $y \in Y$ and where $j_y : X \times y \hookrightarrow X \times Y$ is the natural inclusion. This can be rewritten also as:
 $$\chi(L_1^{\otimes k_1} \otimes \cdots L_p^{\otimes k_p} \otimes {j_y}^*\G) = 0,$$
 for all $L _1, \cdots , L_p \in \Pic(X)$, all $k_1, \cdots , k_p \in \mathbb{N}$ and generic $y \in Y$. Using the Grothendieck-Riemann-Roch Theorem, we find:
 
 \begin{equation*}  
 \int_{X} \ch(L_1^{\otimes k_1} \otimes \cdots  L_p^{\otimes k_p}).\ch(j^*_y \G).td(X) = 0,
 \end{equation*}
 for all $L _1, \cdots , L_p \in \Pic(X)$, all $k_1, \cdots , k_p \in \mathbb{Z}$ and generic $y \in Y$. As a consequence, we get:
 
 \begin{equation} \label{eq1}
\left( c_1(L_1)^{k_1}. \cdots . c_1(L_p)^{k_p}. \ch(j^*_y \G).td(X) \right)_{8-2k} = 0,
 \end{equation}
for all $L _1, \cdots , L_p \in \Pic(X)$, all $0 \leq k \leq 4$, all $k_1, \cdots , k_p \in \mathbb{N}$ such that $k_1 + \cdots + k_p =k$. The Chern character is taken here to be with value in $H^{\bullet}(X, \CC)$. Since numerical equivalence and homological equivalence coincide for curves and divisors, we deduce that:

\begin{equation*}
\left( \ch(j_y^* \G).td(X) \right)_{2k} = 0,
\end {equation*}
for $k=0,1,3,4$.

\bigskip

Let us prove that $\left( \ch(j_y^* \G).td(X) \right)_{4}$ also vanishes. By equation \ref{eq1}, we know that $\left( \ch(j_y^* \G).td(X) \right)_{4}$ is in the primitive cohomology of $X$. Assume that $\left( \ch(j_y^* \G).td(X) \right)_{4} \neq 0$, then the Hodge-Riemann bilinear relations imply that:

\begin{equation*}
\left( \ch(j_y^* \G).td(X) \right)_{4} . \left( \ch(j_y^* \G).td(X) \right)_{4} \neq 0.
\end{equation*}

But $\left( \ch(j_y^* \G).td(X) \right)_{2k} = 0$ for $k=0,1,3,4$, so that $\ch(j_y^* \G)  = \ch(j_y^*\G).td(X).td(X)^{-1}$ has non-vanishing components only in degree $4,6,8$ and its degree $4$ component is $\left( \ch(j_y^* \G).td(X) \right)_{4}$. Hence, we find that $\ch(j_y^* \G^{\vee})$ has non-vanishing components only in degree $4,6,8$ and that its degree $4$ component is $\left( \ch(j_y^* \G).td(X) \right)_{4}$ (here $\G^{\vee}$ is the derived dual of $\G$). We deduce that:

\begin{equation} \label{eq2}
\begin{split}
& \int_{X} \ch \left((j_y^* \G^{\vee})^{\vee}) \right . \ch((j_y^* \G^{\vee}).td(X) \\
= & \left( \ch(j_y^* \G).td(X) \right)_{4}.\left( \ch(j_y^* \G).td(X) \right)_{4} \neq 0. \end{split}
\end{equation}

But the Grothendieck-Riemann-Roch Theorem implies:

\begin{equation*}
\int_{X} \ch \left((j_y^* \G^{\vee})^{\vee}) \right . \ch((j_y^* \G^{\vee}).td(X) = \chi(j_y^* \G^{\vee}, j_y^* \G^{\vee}).
\end{equation*}
 On the other hand the left adjoint to $\Phi : \DB(X) \rightarrow \DB(Y)$ is a Fourier-Mukai transform with kernel $\G^{\vee} \otimes p^* K_Y$. This shows that $j^*_y \G^{\vee} \simeq \Phi^* (\CC(y))$, where $\Phi^*$ is the left adjoint to $\Phi$. Thus, we find that:
 
 \begin{equation*}
  \chi(j_y^* \G^{\vee}, j_y^* \G^{\vee}) = \chi \left(\Phi^*(\CC(y)), \Phi^*(\CC(y)) \right).
 \end{equation*}
We know by hypothesis that $\CC(y) \in \Phi(\DB(X))$, for generic $y \in Y$, and that $\Phi^* : \Phi(\DB(X)) \longrightarrow \DB(X)$ is an equivalence. We conclude that:

\begin{equation*}
\chi \left(\Phi^*(\CC(y)), \Phi^*(\CC(y)) \right) = \chi \left(\CC(y), \CC(y) \right) = 0.
\end{equation*} 
This is a contradiction with equation \ref{eq2} and we this proves that $\left( \ch(j_y^* \G).td(X) \right)_{4} = 0$ for generic $y \in Y$. We deduce that $\ch(j^*_y \G) = 0$ and then that $\ch(j_y^* \G^{\vee}) = 0$ for generic $y \in Y$. This translates as $\ch \left(\Phi^*(\CC(y)) \right) = 0$. But this is impossible. Indeed, $\Phi^* : \Phi(\DB(X)) \longrightarrow \DB(X)$ being an equivalence, we know that it induces a bijection between the image of the Chern character from $\Phi(\DB(X))$ to $H^{\bullet}(Y, \CC)$ and $H^{\bullet}(X,\CC)$. Since $\CC(y) \in \Phi(\DB(X))$ and that its class in $H^{\bullet}(Y,\CC)$ is non-zero, we know that the class of $\Phi^*(\CC(y))$ must also be non-zero.

\bigskip

We conclude that our starting hypothesis is absurd. Thus, there exists $L \in \Pic(X)$ such that the rank of $\Phi(L)$ is non-zero. Using the trace map, we find an injection of graded algebras:

\begin{equation*}
H^{\bullet}(\OO_Y) \hookrightarrow \mathrm{Hom}^{\bullet}(\Phi(L), \Phi(L)).
\end{equation*}

But $\Phi : \DB(X) \hookrightarrow \DB(Y)$ is fully faithful, so that $\mathrm{Hom}^{\bullet}(\Phi(L), \Phi(L)) \simeq \mathrm{Hom}^{\bullet}(L, L) \simeq H^{\bullet}(\OO_X)$. As a consequence, we have an injection of graded algebras:

\begin{equation} \label{eq3}
H^{\bullet}(\OO_Y) \hookrightarrow H^{\bullet}(\OO_X).
\end{equation}

\bigskip
\bigskip
\bigskip
Now we consider the functor $\Phi^* : \DB(Y) \longrightarrow \DB(X)$. The very same proof as above shows that there exists $L \in \Pic(Y)$ such that the rank of $\Phi^*(L)$ is non zero. The trace maps yields again an injection of graded algebras:

\begin{equation*}
H^{\bullet}(\OO_X) \hookrightarrow \mathrm{Hom}^{\bullet}(\Phi^*(L),\Phi^*(L)).
\end{equation*}
On the other hand, given a map $a_k :  L \rightarrow L[k]$ in $\DB(Y)$, we know that there exists unique $b_k \in \mathrm{Hom}(\Phi^* L, \Phi^* L[k])$ and $c_k \in \mathrm{Hom}(\Psi^{!}(L), \Psi^{!}(L)[k])$ such that the diagram:

\begin{equation*}
\xymatrix{ \Psi^{!}(L) \ar[dd]^{c_k} \ar[rr]& & L \ar[rr] \ar[dd]^{a_k} & & \Phi^{*} (L) \ar[dd]^{b_k} \\
& & & & \\
\Psi^{!}L[k] \ar[rr] & & L[k] \ar[rr]  & & \Phi^{*} L[k] \\}
\end{equation*}
commutes (here $\Psi : {}^{\perp} \Phi(\DB(X)) \hookrightarrow \DB(Y)$ is the left-orthogonal to $\Phi(\DB(X))$ in $\DB(Y)$ and $\Psi^{!}$ is the right adjoint to $\Psi$). Using axiom \textbf{TR3} for the definition of a triangulated category, we deduce that the natural map $ \mathrm{Hom}^{\bullet}(L,L) \longrightarrow \mathrm{Hom}^{\bullet}(\Phi^*(L),\Phi^*(L)) $ is an epimorphism of graded vector spaces. Since $ \mathrm{Hom}^{\bullet}(L,L) \simeq H^{\bullet}(\OO_Y)$, we conclude that the injective map of graded algebras in equation \ref{eq3} is actually an isomorphism of graded algebras.

\end{proof}

\end{subsection}

\begin{subsection}{Definition and construction techniques}
We first recall the definition of smoothness, compactness and regularity for triangulated categories.

\begin{defi}[\cite{kontse}, \cite{orlov-nonco}]
Let $\T$ be the derived category of DG-modules over some DG-algebra $(\A,d)$ (over $\mathbb{C}$). The category $\T$ is said to be:

\begin{itemize}
\item \textbf{smooth}, if $\A$ is a perfect bi-module over $\A \otimes_{\mathbb{C}} \A^{op}$.
\item \textbf{compact}, if $ \dim H^{\bullet}(\A,d) < + \infty$.
\item \textbf{regular}, if it has a strong generator.
\item \textbf{Calabi-Yau of dimension $p$} if the shift by $p$ is a Serre functor for $\A$.
\end{itemize} 

\end{defi}

Assume that $\T = \DB(X)$, where $X$ is an algberaic over $\mathbb{C}$. It is easily shown that $X$ is smooth and proper over $\mathbb{C}$ if and only if $\T$ is smooth and compact (see \cite{kontse}). Note also that if $\T$ is a semi-orthogonal component of the derived category of a smooth proper scheme over $\mathbb{C}$, then $\T$ is smooth, compact and regular (see \cite{orlov-nonco}). With these definitions in hand, we can introduce the main notion of this paper:

\begin{defi}[compact hyper-K\"ahler categories] Let $\T$ be a smooth, compact and regular triangulated category which is closed under direct summands. Assume that $\T$ is a semi-orthogonal component of $\DB(\C)$, where $\C$ is an abelian category with a rank function. We say that $\T$ is a \textbf{compact hyper-K\"ahler category (with respect to its embedding in $\DB(\C)$)} if $\T$ is Calabi-Yau of dimension $2m$ and there is a unique homological unit for $\T$ (with respect to its embedding in $\DB(\C)$), which is isomorphic $\mathbb{C}[t]/(t^{m+1})$ with $t$ homogeneous of degree $2$.
\end{defi}

Proposition $A.1$ of \cite{Huy-Nieper} implies the following:

\begin{prop} \label{comtononcom}
Let $X$ be an algebraic variety. The category $\DB(X)$ is compact hyper-K\"ahler (with respect to its embedding in $\DB(X)$) if and only if the variety $X$ is compact hyper-K\"ahler.
\end{prop}

We will see that one can construct many examples of compact hyper-K\"ahler categories which are non-commutative. It seems extremely hard to find new examples of commutative compact hyper-K\"ahler manifolds. Actually, one can produce a lot of compact singular holomorphically symplectic varieties \cite{Mukmuk}. But almost all of them do not admit any geometric crepant resolution of singularities. Hence, I believe that the following result opens the door to a new world of compact hyper-K\"ahler spaces.

\begin{theo} \label{hk-crepant}
Let $Y$ be a projective manifold with Gorenstein rational singularities of dimension $2m$. Assume that $\omega_Y = \OO_Y$ and that $H^{\bullet}(\OO_Y) \simeq \mathbb{C}[t]/(t^{m+1})$, with $t$ homogenous of degree $2$. Any categorical strongly crepant resolution of $Y$ is a compact hyper-K\"ahler category.
\end{theo}

The above statement is slightly ambiguous as we haven't proved that the notion of compact hyper-K\"ahler category is independent of the embedding inside the derived category of an abelian category with a rank function. However, our definition of categorical resolution always refer to a Deligne-Mumford stack which homologically dominates $Y$. In the above statement, we implicitly refer to the embedding of $\T$ inside the derived category of this Deligne-Mumford stack.

\bigskip

\begin{proof}

Let $p : \X \rightarrow Y$ be a projective Deligne-Mumford stack which homologically dominates $Y$ and let $\T \subset \DB(\X)$ be an admissible full subcategory such that the induced map : $\RR p_* : \T \rightarrow \DB(Y)$ is a strongly crepant resolution. Since $\T$ is an admissible subcategory of the derived category of a smooth projective Deligne-Muford stack, we know that $\T$ is smooth, compact and regular. Furthermore, it is a strongly crepant resolution of a Gorenstein projective variety whose dualizing bundle is trivial, hence $\T$ is Calabi-Yau of dimension $\dim Y = 2m$.

We are only left to prove that there is a unique homological unit for $\T$ (with respect to its embedding inside $\DB(\X)$), which is isomorphic to $\mathbb{C}[t]/(t^{m+1})$ with $t$ homogeneous of degree $2$. By hypothesis, we have $\RR p_* \OO_{\X} = \OO_Y$, so that $H^{\bullet}(\OO_{\X}) \simeq H^{\bullet}(\OO_Y) \simeq \mathbb{C}[t]/(t^{m+1})$ (with $t$ homogeneous of degree $2$) is a homological unit for $\DB(\X)$. Hence, for all $\F \in \T$, we have a graded algebra morphism:

\begin{equation*}
 \mathbb{C}[t]/(t^{m+1}) \rightarrow \HHH^{\bullet}(\F, \F),
\end{equation*}
given by $a \rightarrow id_{\F} \otimes a$. As a consequence, this morphism satisfies the functoriality condition stated in definition \ref{homounit}. Furthermore this morphism is split when the rank of $\F$ is not zero. But $\OO_X \in \T$, so that there is a unique homological unit for $\T$ (with respect to its embedding in $\DB(\X)$), which is isomorphic to $\mathbb{C}[t]/(t^{m+1})$ with $t$ homogeneous of degree $2$.
\end{proof}

\begin{cor} \label{finite}
Let $X$ be a compact hyper-K\"ahler variety and $\GG$ be a finite group of symplectic automorphisms of $X$. The category $\DB(Coh^{\GG}(X))$ is a compact hyper-K\"ahler category.
\end{cor}

\begin{proof}
One can show directly that $\DB(Coh^{\GG}(X)$ is a compact hyper-K\"ahler category, but I think it is interesting to show it is a consequence of Theorem \ref{hk-crepant}. Indeed, if $\GG$ is a finite group of symplectic automorphisms of $X$, then $X/ \GG$ is a projective Gorenstein variety with rational singularities. The generator of $H^{2}(\OO_X)$ being $\GG$-equivariant, it descends to $X/G$ and its top wedge-product remains non zero on $X/G$. As a consequence, we have $H^{\bullet}(\OO_{X/G}) \simeq \mathbb{C}[t]/(t^{m+1})$ with $t$ homogeneous of degree $2$ ($m$ is the half-dimension of $X$). Theorem \ref{quotienttheo} implies that $\DB(Coh^{\GG}(X)$ is a strongly crepant resolution of $X/G$. Theorem \ref{hk-crepant} then proves that $\DB(Coh^{\GG}(X))$ is a compact hyper-K\"ahler category.
\end{proof}

\bigskip

There has been recently quite a bit of work on symplectic automorphisms of compact hyper-K\"ahler manifolds (\cite{Chiara, BNWS, mongardi}. Using the existing results in the literature and corollary \ref{finite}, one might hope to construct a vast number of different deformation families of compact hyper-K\"ahler categories in each even dimension. This would put the theory of compact hyper-K\"ahler categories on the same footing as the theory of strict Calabi-Yau manifolds : we still have no efficient tools to classify them, but one can construct a large amount of non-equivalent (up to deformation) examples of such spaces in each fixed dimension. 
\bigskip

\begin{rem} The notion of (holomorphically) symplectic stack has been defined by Pantev, To\"en, Vaqui\'e and Vezzosi \cite{PTVV} and by Zhang \cite{ziyuzhang}. It would be of course desirable to know if one can define the notion of \textit{irreducible holomorphically symplectic stack} and if the derived categories of such stacks are related to compact hyper-K\"ahler categories.
\end{rem}

One of my primary goal when developing the theory of compact hyper-K\"ahler categories was to understand whether the sporadic examples of compact hyper-K\"ahler manifolds discovered by O'Grady could be part of a larger sequence of examples living in the non-commutative world. Let $\mathcal{M}_{K3}(2,0,2r)$ be the moduli space of rank $2$ torsion free sheaves with $c_1 = 0$, $c_2 = 2r$ and which are semi-stable with respect to a generic polarization. Because of the parity of $c_2$, these moduli spaces are not smooth for $r \geq 2$. O'Grady proved that $\mathcal{M}_{K3}(2,0,4)$ admits a crepant resolution and that this crepant resolution is a compact hyper-K\"ahler manifold which is not deformation equivalent to the previously known examples of compact hyper-K\"ahler manifolds \cite{OG1}.

\bigskip

It was then proved in \cite{Choy-Kiem} and \cite{KLS} that the moduli spaces $\mathcal{M}_{K3}(2,0,2r)$ do not admit any crepant resolution for $r \geq 3$. Hence one can't hope to find new examples of commutative compact hyper-K\"ahler variety starting with these moduli spaces. However, it seems quite likely that these moduli spaces have categorical crepant resolutions. Exhibiting such resolutions would provide a whole new heap of compact hyper-K\"ahler categories. This would also demonstrate that the O'Grady examples are not sporadic at all : they would be part of a series which naturally lives in the non-commutative world.

\begin{quest}
Let $r \geq 3$ be an integer. Does the moduli spaces $\mathcal{M}_{K3}(2,0,2r)$ admit a categorical strongly crepant resolution of singularities? 
\end{quest}
\end{subsection}

\begin{subsection}{Deformation theory for compact hyper-K\"ahler categories}

In this subsection, I will prove some basic results for the deformation theory of compact hyper-K\"ahler categories. They will be used in the last section of this paper to prove that there exists compact hyper-K\"ahler categories of dimension $4$ which deformations are never equivalent to the derived category of a projective variety. 

\bigskip

I will focus on a specific type of deformation of triangulated categories : deformation inside the derived category of an algebraic variety (all results proven below should carry on without any problem to deformation inside the derived category of a Deligne-Mumford stack). Let $\T \subset \DB(X)$ be a full admissible subcategory. Given a smooth algebraic variety $B$, one wants to define the deformation of $\T$ inside $\DB(X)$ over $B$.

\begin{defi}
Let $X$ be a smooth projective variety, let $\T \subset \DB(X)$ be a full admissible subcategory and $B$ a \textbf{smooth connected algebraic variety} with a marked point $0 \in B$. A \textbf{smooth deformation of} $\T$ \textbf{inside $X$ over $B$} is the data of:
\begin{itemize}
\item a smooth projective morphism $\pi : \X \rightarrow B$ such that $\X_0 = X$,

\item a full admissible subcategory $\D \subset \DB(\X)$, which is $B$-linear, such that $ \EE_0 := \EE \otimes_{\OO_{\X \times_{B} \X}} \OO_{\X_0 \times \X_0} \in \DB(\X_0 \times \X_0)$ is the kernel of the projection $\DB(\X_0) \rightarrow \T$, where $\EE \in \DB(\X \times_{B} \X)$ is the kernel representing the projection functor $\DB(\X) \rightarrow \D$.

\end{itemize}
\end{defi}

The existence of the kernels in the above definition has been proved by Kuznetsov in \cite{kuz6}. We have a semi-orthogonal decomposition $\DB(\X) = \langle \D, {}^{\perp} \D \rangle$ and I denote by ${}^{\perp} \EE \in \DB(\X \times_{B} \X)$ the kernel of the projection $\DB(\X) \rightarrow {}^{\perp} \D$. Let us display a Cartesian diagram which will be important to study the deformation of $\T$ over $B$.

\begin{equation*}
\xymatrix{ & &  \ar[lldd]_{p} \X \times_{B} \X \ar[rrdd]^{q} & &  \\
& & & & \\
\X & & \ar[lldd]_{p_b} \X_b \times \X_b \ar[uu]^{j_b} \ar[rrdd]^{q_b} & & \X \\
& & & & \\
\X_b \ar[uu]^{i_b} & & & & \X_b \ar[uu]^{i_b} }
\end{equation*}

\begin{prop} \label{defdecompo}
With hypotheses and notation as above, for all $b \in B$, there exists a semi-orthogonal decomposition:

\begin{equation*}
\DB(\X_b) = \langle \D_b,{}^{t} \D_b \rangle, 
\end{equation*}
where $\D_b$ (resp. ${}^{t} \D_b$) is the full subcategory of $\DB(\X_b)$ closed under taking direct summands which is generated by the objects $\RR {p_b}_* (\LL q_b^* \F \otimes \EE_b)$ (resp. $\RR {p_b}_* (\LL q_b^* \F \otimes {}^{\perp} \EE_b)$), for $\F \in \DB(X_b)$. 
\end{prop}

\noindent This proposition allows one to think of the $\D_b$ for $b \in B$ as the deformation of $\D_0 = \T$ over $B$.

\bigskip

\begin{proof}
Since $\X \rightarrow B$ is projective, the family of line bundles $\OO_{\X/B}(m)|_{\X_b}, m \in \mathbb{N}$ generates $\DB(\X_b)$. Hence, the category $\D_b$ (resp. ${}^{t} \D_b$) is also generated by the $\RR {p_b}_* (\LL q_b^* \LL i_b ^* \F \otimes \EE_b)$ (resp. $\RR {p_b}_* (\LL q_b^* \LL i_b ^* \F \otimes {}^{\perp} \EE_b)$), for $\F \in \DB(\X)$. I first prove that ${}^{t} \D_b$ is left orthogonal to $\D_b$. Let $\F, \G \in \DB(\X)$, we have:

\begin{equation*}
\begin{split}
& \HHH (\RR {p_b}_* (\LL q_b^*(\LL i_b^* \F) \otimes {}^{\perp} \EE_b), \RR {p_b}_* (\LL q_b^*(\LL i_b^* \G) \otimes \EE_b)) \\
& = \HHH (\RR {p_b}_* (\LL j_b^*( \LL q^* \F \otimes {}^{\perp} \EE)), \RR {p_b}_* (\LL j_b^*( \LL q^* \G \otimes \EE)))\\
& = \HHH (\LL i_b^* (\RR p_*( \LL q^* \F \otimes {}^{\perp} \EE)), \LL i_b^* (\RR p_* ( \LL q^* \G \otimes \EE)))\\
& = \HHH (\RR p_*( \LL q^* \F \otimes {}^{\perp} \EE),\RR {i_b}_*( \LL i_b^* (\RR p_* ( \LL q^* \G \otimes \EE)))) \\
&= \HHH ( \RR p_*( \LL q^* \F \otimes {}^{\perp} \EE), \RR p_* ( \LL q^* \G \otimes \EE) \otimes \RR {i_b}_* \OO_{\X_b}),
\end{split}
\end{equation*}
here the first equality is the identity $\LL q_b^* \LL i_b^* =\LL j_b^* \LL q^*$, the second is the flat base change $\RR {p_b}_* \LL j_b^* = \LL i_b^* \RR p_*$, the third is adjunction with respect to $i_b$ and the fourth is the projection formula with respect to $i_b$. By flat base change for the morphism $\pi : \X \rightarrow B$, we have $\RR {i_b}_* \OO_{\X_b} = \LL \pi^* \mathbb{C}(b)$. The category $\D$ is $B$-linear by hypothesis, so that $\RR p_* ( \LL q^* \G \otimes \EE) \otimes \RR {i_b}_* \OO_{\X_b} \in \D$. As a consequence, we deduce the vanishing:

\begin{equation*}
\HHH (\RR p_*( \LL q^* \F \otimes {}^{\perp} \EE), \RR p_* ( \LL q^* \G \otimes \EE) \otimes \RR {i_b}_* \OO_{\X_b}) = 0.
\end{equation*}
As $\D_b$ (resp. ${}^{t} \D_b$) is the full subcategory of $\DB(\X_b)$ closed under taking direct summands which is generated by the $\RR {p_b}_* (\LL q_b^* \LL i_b ^* \F \otimes \EE_b)$ (resp. $\RR {p_b}_* (\LL q_b^* \LL i_b ^* \F \otimes {}^{\perp} \EE_b)$) for $\F \in \DB(\X)$, the above vanishing finally proves that $\HHH(\F, \G) = 0$, for all $\G \in \D_b$ and $\F \in {}^{t} \D_b$.

\bigskip

We are left to show that for all $\H \in \DB(X_b)$, there exists an exact triangle:

\begin{equation*}
\G \rightarrow \H \rightarrow \F,
\end{equation*}

with $\F \in {}^{t} \D_b$ and $\G \in \D_b$. But on $\X \times_B \X$, we have an exact triangle:

\begin{equation*}
\EE \rightarrow \OO_{\Delta/B} \rightarrow {}^{\perp} \EE.
\end{equation*}
Hence, for all $\F \in \DB(X_b)$, we have an exact triangle:

\begin{equation*}
\RR {p_b}_* \LL (q_b^*(\F) \otimes \EE_b) \rightarrow \F \rightarrow \RR {p_b}_* \LL (q_b^*(\F) \otimes {}^{\perp} \EE_b).
\end{equation*}
\end{proof}

\begin{cor} \label{kernel}
For all $b \in B$, the objects $\EE_b$ (resp. ${}^{\perp} \EE_b$) is the kernel of the projection functor $\DB(\X_b) \rightarrow \D_b$ (resp. $\DB(\X_b) \rightarrow {}^{\perp} \D_b$).
\end{cor}

\begin{proof}
Using exactly the same identities as in the proof of proposition \ref{defdecompo}, one shows that $\RR {p_b}_*(\LL q_b^*(\G) \otimes \EE_b) $ is quasi-isomorphic to $\G$ if $\G \in \D_b$ and is zero if $\G \in {}^{t} \D_b$ (the opposite holds for ${}^{\perp} \EE_b$). Since Proposition \ref{defdecompo} shows that ${}^{t} \D_b = {}^{\perp} \D_b$, the claim is proved.
\end{proof}

\begin{theo} \label{semiconti}
Let $X$ be a smooth projective variety and $\T \subset \DB(X)$ a full admissible subcategory. Let $B$ a smooth variety and $\D$ be a smooth deformation of $\T$ over $B$. The dimension of the Hochschild homology of $\D_b$ is constant for $b \in B$.
\end{theo}

Note that we do not need to assume that the kernel of the projection $\DB(\X) \rightarrow \D$ is flat over $B$.

\begin{proof}
Let $\X \rightarrow B$ be a smooth projective morphism such that $\D$ is a full admissible subcategory of $\DB(\X)$. Let $\EE \in \DB(\X \times_B \X)$ be the kernel of the projection functor $\DB(\X) \rightarrow \D$. By corollary \ref{kernel}, we know that for all $b \in B$, the object $\EE_b \in \DB(\X_b \times \X_b)$ is the kernel of the projection functor $\DB(\X_b) \rightarrow \D_b$. As a consequence of Theorem 4.5 in \cite{kuz3}, we have an equality:

\begin{equation*}
\HH_{\bullet}(\D_b) = {H}^{\bullet}(\X_b \times \X_b, \EE_b \otimes \EE_b^{T}),
\end{equation*}
where $\EE_b^{T}$ is the pull back of $\EE_b$ with respect to the permutation $\X_b \times \X_b \rightarrow \X_b \times \X_b$. Let us prove that the dimension of the cohomology vector spaces $H^{i}(\X_b \times \X_b, \EE_b \otimes \EE_b^{T})$ are upper semi-continuous with respect to $b \in B$ for all $i$. By flat base change for the diagram:

\begin{equation*}
\xymatrix{ \X_b \times \X_b \ar[rr]^{j_b} \ar[dd]^{\pi_b} & & \X \times_B \X \ar[dd]^{\pi}  \\
& &  \\
\mathrm{Spec} (\mathbb{C}(b)) \ar[rr] & & B }
\end{equation*}
we have the equality:
\begin{equation*}
{H}^{i}(\X_b \times \X_b, \EE_b \otimes \EE_b^{T}) = \mathcal{H}^i(\RR \pi_*(\EE \otimes \EE^{T}) \otimes \mathbb{C}(b)).
\end{equation*}

Since $B$ is a smooth variety, we can represent $\RR \pi_*(\EE \otimes \EE^{T}) \otimes \mathbb{C}(b)$ by a bounded complex of vector bundles on $B$, say $E^{\bullet}$. Thus, we only have to show the following : the cohomology sheaves of $E^{\bullet} \otimes \mathbb{C}(b)$ are upper semi-continuous, for $b \in B$. This result is now obvious as the dimension of the image of the differential:

\begin{equation*}
d^{\bullet} \otimes \mathbb{C}(b) : E^{\bullet} \otimes \mathbb{C}(b) \rightarrow E^{\bullet +1} \otimes \mathbb{C}(b)
\end{equation*}
is lower semi-continuous with respect to $B$.

\bigskip

We have proved that the dimension of $\HH_{i}(\D_b)$ is upper semi-continuous with respect to $B$, for all $i$. This holds also true for the dimension $\HH_{i}({}^{\perp} \D_b)$. By corollary 7.5 of \cite{kuz3}, we have:

\begin{equation*}
\HH_{i}(\D_b) \oplus \HH_{i}({}^{\perp}\D_b) = \HH_{i}(\DB(\X_b)).
\end{equation*}
But the morphism $\X \rightarrow B$ is smooth projective, so that the Hodge numbers of $\X_b$ are constant with respect to $B$. By the Hochschild-Kostant-Rosenberg decomposition, this implies that the Hochschild numbers of $\X_b$ are constant. Hence the sum of the dimensions of the cohomology vector spaces $\HH_{i}(\D_b)$ and $\HH_{i}({}^{\perp}\D_b)$ is constant with respect to $B$. But each dimension is upper semi-continuous with respect to $B$, so that they are in fact both constant with respect to $B$.

\end{proof}

Before going turning to deformation results for compact hyper-K\"ahler categories, I want to comment about the level of generality of the deformation theory used above.  In order to define the notion of deformation of an admissible subcategory $\T \subset \DB(X)$, one could be tempted to work with a seemingly more general definition, as follows. A deformation of $\T$ over $B$ is the data of a (not necessarily flat) morphism $\pi : \X \rightarrow \B$ and a $B$-linear admissible subcategory $\D \subset \DB(\X \times_{B} \X)$, such that $\D_0 = \T$ and the flat base change formula holds for $\D$ with respect to the diagram:

\begin{equation*}
\xymatrix{ \D_b \ar[rr]^{\RR {j_b}_*}  & & \D   \\
& &  \\
\ar[uu]^{\LL \tilde{\pi}_b^*} \mathrm{Spec} (\mathbb{C}(b)) \ar[rr]^{{k_b}_*} & & B \ar[uu]^{\LL \tilde{\pi}^*}} 
\end{equation*}

The base change formula would imply $\RR {j_b}_* \LL \tilde{\pi}_b^* \mathbb{C}(b) = \LL \tilde{\pi}^* \RR {k_b}_* \mathbb{C}(b)$. But we have a commutative diagram:

\begin{equation*}
\xymatrix{ \D_d \ar[rr] \ar[rrdd]^{\RR {\tilde{\pi_b}_*}} & & \DB(\X_b) \ar[dd] \\
& &  \\
 & & \DB(\mathrm{Spec}(\mathbb{C}(b))}
\end{equation*}
Assume that $\OO_{\X} \subset \D$. The fact that $\D$ is $B$-linear then implies $\LL \tilde{\pi}_b^* \mathbb{C}(b) = \OO_{X_b}$. As a consequence, we have $\LL \tilde{\pi}^* \RR {k_b}_* \mathbb{C}(b) = {j_b}_*\OO_{X_b}$. In particular, we have $Tor_B^{1}(\OO_{\X}, \mathbb{C}(b)) = 0$. By Theorem $22.3$ of \cite{matsumura}, the morphism $\X \rightarrow B$ is flat. Hence, a strictly more general setting than the one developed above for the deformation of triangulated categories \textbf{can not} be obtained if one requires the following three conditions:

\begin{itemize}
\item  the total space of the deformation is a full admissible subcategory of the derived category of an algebraic variety,

\item the base change formula holds for the total space of the deformation,

\item $\OO_{\X} \in \D$.
\end{itemize}

The first two conditions seem essential if one wants to get some significant homological results while working with admissible subcategories of derived categories of algebraic varieties. As far as the third condition is concerned, it is satisfied in many examples (for instance in the setting of non-commutative resolution of singularities).
\bigskip

I will now focus on the deformation theory of compact hyper-K\"ahler categories. We start with the following:

\begin{lem} \label{defstructural}
Let $X$ be a smooth projective variety, let $\T \subset \DB(X)$ be a full admissible subcategory and $B$ a smooth algebraic variety. Let $\D$ be a smooth deformation of $\T$ over $B$ with respect to $\pi : \X \longrightarrow B$. Assume that $\OO_X \in \T$. Then, there exists an open $0 \in U \subset B$, such that $\OO_{\X_b} \in \D_b$, for all $b \in U$. 
\end{lem}

\begin{proof}
Let $\EE$ be the kernel giving the projection functor $\DB(\X) \longrightarrow \D$. The hypothesis $\OO_{\X_0} = \OO_X \in \D_0$ can be translated as:
\begin{equation*}
 \RR {p_0}_* (\LL q_0^* \OO_{X} \otimes {}^{\perp} \EE_0) = 0,
\end{equation*}
that is:
\begin{equation*}
 \RR {p_0}_* ( {}^{\perp} \EE_0) = 0.
\end{equation*}
By semi-continuity, there exists an open neighborhood $0 \in U \subset B$ such that:

\begin{equation*}
 \RR {p_b}_* ( {}^{\perp} \EE_b) = 0,
\end{equation*}
for all $b \in U$. Thus, we have $\OO_{\X_b} \in \D_b$, for all $b \in U$.
\end{proof}
\bigskip

The very same proof also yields:

\begin{lem} \label{defpoints}
Let $X$ be a smooth projective variety, let $\T \subset \DB(X)$ be a full admissible subcategory and $B$ a smooth algebraic variety. Let $\D$ be a smooth deformation of $\T$ over $B$ with respect to $\pi : \X \longrightarrow B$. Assume that $\CC(x) \in \T$, for generic $x \in X$. Then, there exists an open $0 \in U \subset B$, such that $\CC(x_b) \in \D_b$, for generic $x_b \in \X_b$ and for all $b \in U$. 
\end{lem}

We now state our first result on the deformation theory of hyper-Kähler categories. It shows that a ``small deformation'' of a hyper-Kähler category is still hyper-Kähler.

\begin{prop} \label{smalldef}
Let $X$ be a smooth projective variety, let $\T \subset \DB(X)$ be a full admissible subcategory and $B$ a smooth algebraic variety. Let $\D$ be a smooth deformation of $\T$ over $B$ with respect to $\pi : \X \longrightarrow B$. Assume that $\OO_{X} \in \T$ and that $\T$ is a compact hyper-K\"ahler category (with respect to its embedding in $\DB(X)$). Then, there exists a neighborhood $0 \in U \subset B$, such that $\D_b$ is compact hyper-K\"ahler (with respect to its embedding in $\DB(\X_b)$) for all $b \in U$.
\end{prop}

\begin{proof}
Let $ \pi : \X \rightarrow B$ be the smooth projective morphism in which the deformation $\D$ is embedded. We know that $\D_0 = \T$ is compact hyper-K\"ahler of dimension $2m$ (with respect to its embedding in $\DB(X)$). In particular, the category $\T$ is Calabi-Yau of dimension $2m$. Hence there exists a quasi-isomorphism:

\begin{equation*}
\theta_0 : \EE_0 \rightarrow \EE_0 \otimes q_0^* \omega_{\X_0}[\dim X_0-2m].
\end{equation*} 

But $\EE_0 = \EE \otimes_{\X \times_B \X}\OO_{\X_0 \times \X_0}$ and $\omega_{\X_0} = \omega_{\X/B} \otimes_{\X \times_B \X}\OO_{\X_0 \times \X_0}$. Hence, by Nakayama's lemma, there exists a neighborhood $0 \in U \subset B$, such that $\theta_0$ can be lifted to a quasi-isomorphism:

\begin{equation*}
\theta_U : \EE \otimes_{\X \times_B \X}\OO_{U} \rightarrow \EE \otimes_{\X \times_B \X} q^* \omega_{\X_U/U}[\dim(\X_U/U)-2m].
\end{equation*}
This proves that the categories $\D_b, b \in U$ are Calabi-Yau of dimension $2m$. Since $X_b$ is smooth projective for all $b \in B$, the categories $\D_b$ are also smooth, compact and regular for all $b \in B$. It remains to prove (up to shrinking $U$), that $\mathbb{C}[t]/(t^{m+1})$ (with $t$ homogeneous of degree $2$) is a homological unit for $\D_b, b \in U$. 

\bigskip

We know by hypothesis that $\T$ contains $\OO_{X}$. Hence, by lemma  \ref{defstructural}, there exists an open $0 \in U' \subset U$ such that the categories $\D_b, b \in U'$ all contain $\OO_{\X_b}$. We deduce that for all $b \in U'$, the graded algebra $H^{\bullet}(\OO_{\X_b})$ is a homological unit for $\D_b$. But $H^{\bullet}(\OO_{X_0}) \simeq \mathbb{C}[t]/(t^{m+1})$. Hence there exists another neighborhood $0 \in U'' \subset U'$ such that $H^{\bullet}(\OO_{X_b}) \simeq \mathbb{C}[t]/(t^{m+1})$, for all $b \in U''$. The open $U''$ is the neighborhood of $0 \in B$ we are looking for.

\end{proof}

The above statement shows that being compact hyper-K\"ahler is an open condition (if one assumes that $\OO_{\X} \in \D$). I also expect it to be a closed condition. Namely:

\begin{conj} \label{def-HK}
Let $X$ be a smooth projective variety, let $\T \subset \DB(X)$ be a full admissible subcategory and $B$ a smooth algebraic variety. Let $\D$ be a deformation of $\T$ over $B$. Assume that $\D_b$ is compact hyper-K\"ahler for all $b \neq 0$. Then, the category $\D_0 = \T$ is compact hyper-K\"ahler.
\end{conj}

The commutative specialization of this result is well-known. Namely, let $\pi : \X \rightarrow B$ be a smooth projective morphism with $B$ smooth. If $\X_b$ is compact hyper-K\"ahler for all $b \neq 0$, then $\X_0$ is also compact hyper-K\"ahler. It is usually proved using the holonomy principle and the invariance of holonomy groups in smooth families (see \cite{huy2}, section 1). As far as I am aware, there are no algebraic proof of this result. Hence, a proof of conjecture \ref{def-HK}, would certainly require the design of interesting new categorical techniques.

Two key results are to be proved in order to demonstrate conjecture \ref{def-HK} : the invariance of the Calabi-Yau condition and of the homological unit under smooth deformations.

\begin{conj} \label{def-CY}
Let $X$ be a smooth projective variety, let $\T \subset \DB(X)$ be a full admissible subcategory and $B$ a smooth algebraic variety. Let $\D$ be a deformation of $\T$ over $B$. Assume that $\D_b $ is Calabi-Yau of dimension $r$ for all $b \neq 0$. Then, the category $\D_0 = \T$ is Calabi-Yau of dimension $r$.
\end{conj}

Note that it is very unlikely that this conjecture can be proved by abstract algebraic arguments. Indeed, the work of Keller (\cite{keller}) suggests that strong additional hypotheses are usually used in order to prove that a deformation of a Calabi-Yau algebra is again Calabi-Yau. Hence, the fact that the categories appearing in conjecture \ref{def-CY} are subcategories of derived categories of algebraic varieties will certainly play an important role in a potential proof.

\bigskip

Let us conclude this section with a ``long-time'' deformation result for hyper-K\"ahler categories. It gives a partial answer to Conjecture \ref{def-HK} in the four-dimensional case.

\begin{prop} \label{def-hk4}
Let $X$ be a smooth projective variety, let $\T \subset \DB(X)$ be a full admissible subcategory which is Calabi-Yau of dimension $4$ and let $B$ a smooth algebraic variety. Let $\D$ be a smooth deformation of $\T$ over $B$ with respect to $\pi : \X \rightarrow B$. Assume that $\OO_{X} \in \T$ and that for all $b \neq 0$, the category $\D_b$ is compact hyper-K\"ahler of dimension $4$ (with respect to its embedding in $\DB(\X_b)$). Then, the category $\D_0 = \T$ is compact hyper-K\"ahler of dimension $4$ (with respect to its embedding in $\DB(\X_0)$).
\end{prop}

\begin{proof}
We already know that $\T$ is smooth, compact, regular and Calabi-Yau of dimension $4$. Since $\OO_{X} \in \T$, lemma \ref{defstructural} implies that there exists an open subset $0 \in U \subset B$ such that $\OO_{\X_b} \in \D_b$, for all $b \in U$. Hence, for all $b \in U$, the algebra $H^{\bullet}(\OO_{X_b})$ is a homological unit for $\D_b$. By hypothesis, for all $b \neq 0 \in B$, there exists a unique homological unit for $\D_b$ (with respect to its embedding in $\DB(\X_b)$) which is $\mathbb{C}[t]/(t^3)$, with $t$ in degree $2$. As a consequence, for all $b \neq 0 \in U$, we have:

\begin{equation*}
H^{\bullet}(\OO_{\X_b}) = \mathbb{C}[t]/{t^3},
\end{equation*}
with $t$ in degree $2$. Hodge numbers are invariant in smooth family, so that $H^{\bullet}(\OO_{\X_0}) \simeq \mathbb{C}[t]/(t^3)$ as a graded vector space. But the category $\D_0$ is Calabi-Yau of dimension $4$, hence the pairing:

\begin{equation*}
H^{2}(\OO_{\X_0}) \times H^{2}(\OO_{\X_0}) \rightarrow H^{4}(\OO_{\X_0}) \simeq \mathbb{C} 
\end{equation*}
given by the Yoneda product coincide with the Serre-duality pairing : it is non degenerate. As $\dim H^2(\OO_{X_0}) = 1$, we find an isomorphism of graded algebras : $H^{\bullet}(\OO_{\X_0}) \simeq \mathbb{C}[t]/(t^3)$, with $t$ in degree $2$. This proves that $\D_0$ is compact hyper-K\"ahler (with respect to its embedding in $\DB(\X_0)$). 
\end{proof}

One would obviously like to generalize this result in higher dimension. However, the non-degeneracy of the Serre-duality pairing does not have so strong consequences in higher dimensions.
\end{subsection}

\end{section}

\begin{section}{Non-commutative Hilbert schemes}
 In this section, we will be interested in the following question that was asked to me by Misha Verbitsky:
 
 \begin{quest} \label{quest1}
 Let $S$ be a K3 surface and $n \geq 1$. For which subgroups $\GG \subset \mathcal{S}_n$ does the quotient $S \times \cdots \times S/ \GG$ has a categorical crepant resolution 
 of singularities which is a hyper-Kähler category?
 \end{quest}

The commutative version of this question has been solved by Verbitsky himself. Indeed, in \cite{verbi1} he proves the following:

\begin{theo} \label{verbi} Let $S$ be a K3 surface and $n \geq 1$. Let $\GG$ be a subgroup of $\mathcal{S}_n$ such that $S \times \cdots \times S/ \GG$ has a crepant resolution which is hyper-Kähler. Then $\GG = \mathcal{S}_n$ and the resolution is the Hilbert scheme of $n$ points on $S$.
\end{theo}

Note that the result actually proved by Verbitsky in \cite{verbi1} is more general. He shows that if $V$ is a symplectic vector space and $\GG \subset \mathrm{S}p(V)$ is such that $V /\!\!/ \GG$ admits a symplectic resolution, then $\GG$ is generated by symplectic reflections. In the setting of Theorem \ref{verbi}, symplectic reflections are immediately seen to be transpositions. Furthermore, for the crepant resolution of $S \times \cdots \times S/ \GG$ to be irreducible symplectic, we need that each factor of $S \times \cdots S$ is acted on non-trivially by an element of $\GG$. This is easily demonstrates that $\GG = \mathcal{S}_n$.

We now state the answer to question \ref{quest1}:

\begin{theo} \label{nchilbert}
Let $S$ be a K3 surface and $n \geq 1$. Let $\GG \subset \mathcal{S}_n$ acting on $S \times \cdots \times S$ by permutations. The quotient $S \times \cdots \times S/ \GG $ admits a categorical crepant resolution which is a hyper-Kähler category if and only if $\GG$ is one of the following:

\begin{itemize}
\item $\GG = \mathcal{S}_n$,

\item $\GG = \mathcal{A}_n$  (the alternating group),

\item $n=5$ and $\GG = \mathbb{F}_5^*$,

\item $n= 6$ and $\GG = \mathbb{P}GL_2(\mathbb{F}_5)$, 

\item $n=9$ and $\GG = \mathbb{P}GL_2(\mathbb{F}_8)$,

\item $n=9$ and $\GG = \mathbb{P}GL_2(\mathbb{F}_8) \rtimes \mathrm{Gal}(\mathbb{F}_8/\mathbb{F}_2)$.

\end{itemize}
Furthermore, this categorical resolution is always non-commutative in the sense of Van-den-Bergh.
\end{theo}

The categorical McKay correspondence \cite{BKR} implies that:

$$\DB(\mathrm{Hilb}^{[n]}(S)) \simeq \DB(Coh^{\mathcal{S}_n}(S \times \cdots \times S).$$ 
Hence, in case $\GG = \mathcal{S}_n$, our result does not produce any new hyper-Kähler category. On the other hand, for $\GG  = \mathcal{A}_n, \mathbb{F}_5^*, \mathbb{P}GL_2(\mathbb{F}_5), \mathbb{P}GL_2(\mathbb{F}_8), \mathbb{P}GL_2(\mathbb{F}_8) \rtimes \mathrm{Gal}(\mathbb{F}_8/\mathbb{F}_2)$, it seems that Theorem \ref{nchilbert} is the first instance of a result which connects hyper-Käher geometry with the quotient spaces $S \times \cdots \times S/ \GG$.

\bigskip

\begin{proof}
 We denote by $Y$ the quotient $S \times \cdots S / \GG$. By \cite{quotient}, we know that $\DB(Coh^{\GG}(S \times \cdots \times S))$ is a categorical strongly crepant resolution of singularities of $Y$. Furthermore, the Serre functor of the category $\DB(Coh^{\GG}(S \times \cdots \times S))$ is the shift by $2n$ and its homological units is $H^{\bullet}(\OO_{S \times \cdots \times S})^{\GG}$.

 By the Künneth formula, we have $H^{\bullet}(\OO_{S \times \cdots \times S}) = H^{\bullet}(\OO_S) \otimes \cdots \otimes H^{\bullet}(\OO_S)$. Hence, the homological unit of $\DB(Coh^{\GG}(S \times \cdots \times S))$ is :
 
 \begin{equation*}
 \mathfrak{T}^{\bullet} = \left(H^{\bullet}(\OO_S) \otimes \cdots \otimes H^{\bullet}(\OO_S)\right)^{\GG},
 \end{equation*}

where $\GG$ acts on $H^{\bullet}(\OO_S) \otimes \cdots \otimes H^{\bullet}(\OO_S)$ by permutation. Since $H^{\bullet}(\OO_S)$ is generated by $1$ in degree $0$ and by $\sigma_S$ in degree $2$, a basis of $H^k(\OO_{S \times \cdots \times S})$ is given by:

 $$\{  \sigma_S^{ I} \otimes 1^{\{1, \cdots, n \} \setminus I}, \, \, \textrm{for} \, I \in \mathrm{P}_k(\{1, \cdots, n \}) \},$$
 
 where $\sigma_S^{ I} \otimes 1^{\{1, \cdots, n \} \setminus I}$ is the tensor product of $\sigma_S$ in the positions indexed by $I$ and $1$ in the positions by the complement of $I$ in $\{1, \cdots, n\}$ and $\mathrm{P}_k(\{1, \cdots, n \})$ is the set of unordered sets of length $k$ in $\{1, \cdots, n \}$.

\bigskip

The subalgebra of $H^{\bullet}(\OO_S) \otimes \cdots \otimes H^{\bullet}(\OO_S)$ generated by $\sum_{i \in \{1,\cdots,n\}} \sigma_S^{\{i\}} \otimes 1 ^{\{1, \cdots, n \} \setminus    \{i\} }$ is isomorphic to $\mathbb{C}[t]/{t^{n+1}}$, with $t$ in degree $2$ and is invariant by $\GG$ for any $\GG \subset \mathcal{S}_n$. Hence, the category $\DB(Coh^{\GG}(S \times \cdots \times S))$ is hyper-Kähler if and only if $\left(H^{\bullet}(\OO_S) \otimes \cdots \otimes H^{\bullet}(\OO_S)\right)^{\GG}$ is equal to the algebra generated by $\sum_{i \in \{1,\cdots,n\}} \sigma_S^{\{i\}} \otimes 1 ^{\{1, \cdots, n \} \setminus    \{i\} }$. That is if and only if:

$$ \sum_{I \in \mathrm{P}_k(\{1,\cdots,n\})} \sigma_S^{ I} \otimes 1^{\{1, \cdots, n \} \setminus I}$$

is the only element (up to scalar) of degree $k$ in $H^{\bullet}(\OO_S) \otimes \cdots \otimes H^{\bullet}(\OO_S)$ which is invariant under the action of $\GG$, for all $1 \leq k \leq n$. This condition can be rephrased as saying that $\GG$ acts transitively on $\mathrm{P}_k(\{1 \cdots n \})$ for all $1 \leq k \leq n$. Such groups have been classified by Beaumont and Peterson \cite{beaumontpeterson, livingstonewagner} and they are the following:

\begin{itemize}
\item $\GG = \mathcal{S}_n$,

\item $\GG = \mathcal{A}_n$  (the alternating group),

\item $n=5$ and $\GG = \mathbb{F}_5^*$,

\item $n= 6$ and $\GG = \mathbb{P}GL_2(\mathbb{F}_5)$, 

\item $n=9$ and $\GG = \mathbb{P}GL_2(\mathbb{F}_8)$,

\item $n=9$ and $\GG = \mathbb{P}GL_2(\mathbb{F}_8) \rtimes \mathrm{Gal}(\mathbb{F}_8/\mathbb{F}_2)$.

\end{itemize}

\bigskip

In \cite{quotient}, we showed that if $X$ is a smooth projective variety and $\GG$ a reductive algebraic group acting linearly on $X$ with finite stabilizers, then there exists a sheaf of algebra $\B$ on $X /\!\!/ \GG$ such that $\DB(Coh^{\GG}(X)) \simeq \DB(X /\!\!/ \GG, \B)$. This result applies in the present setting and shows that $\DB(Coh^{\GG}(S \times \cdots \times S))$ is a non-commutative crepant resolution of $S \times \cdots \times S / \GG$ in the sense of Van den Bergh which is a hyper-Kähler category of dimension $2n$.

\end{proof}

\bigskip

If $X$ is a smooth projective holomorphically symplectic variety, then the twisted Hochschild-Kostant-Rosenberg isomorphism shows that the Betti cohomology ring of $X$ is isomorphic (as a ring!) to the Hochschild cohomology ring of $\DB(X)$. On the other hand, the Betti cohomology ring of $\mathrm{Hilb}^{[n]}(S)$ has been extensively studied and many fascinating connections between hyper-Kähler geometry, representation theory and number theory have been discovered this way. We refer to the ICM talk of Göttsche for a nice overview of these connections \cite{gottsche}. It wouldn't be surprising that the study of the Hochschild cohomology rings of the categories appearing in Theorem \ref{nchilbert} yield new connections between these three topics. In particular, we feel it is worth asking the following questions:

\begin{quest} \label{q1} Let $\GG$ be any of the groups appearing in Theorem \ref{nchilbert}. Let $\B_{\GG}$ the sheaf of algebra on $S \times \cdots \times S/ \GG$ such that $\DB(Coh^{\GG}(S \times \cdots \times S)) \simeq \DB(S \times \cdots \times S / \GG, \B_{\GG})$. Is $\B_{\GG}$ a sheaf of symplectic algebras in the sense of Ginzburg?
\end{quest}

We feel that this first question shouldn't be too hard and shall be just a matter of checking that Ginzburg's definition \cite{ginzburg} of symplectic algebras matches with ours when the triangulated category under study is the derived category of coherent modules over a sheaf of finitely generated algebras.

\begin{quest} \label{q2} What are the Hochschild numbers of the hyper-Kähler categories appearing in Theorem \ref{nchilbert}? Is it possible to compute explicitly the ring structure of the Hochschild cohomology ring of these examples?
\end{quest}

Following work of Baranovsky \cite{baranovsky}, Arinkin, Hablicsek and Caldararu proved a version of the Hochschild Kostant Rosenberg for global finite quotient stacks \cite{ACH}. The compatibility of this isomorphism (or a twisted version of it) with cup product on both sides is still conjectural. Nonetheless, one can be confident that it will be proven soon. Hence, question \ref{q2} basically boils down to computing the orbifold cohomology of the Deligne-Mumford stack $[S \times \cdots \times S /\!\!/ \GG]$ for all $\GG$ appearing in Theorem \ref{nchilbert}.

\begin{quest} \label{q3}
Does the sum $\sum_{n \geq 0} e(\DB(Coh^{\mathcal{A}_n}(S \times \cdots \times S))) z^n$ have special modularity properties?
\end{quest}

\begin{quest} \label{q4}
Is the graded vector space $\bigoplus_{n \geq 0} \HH^{\bullet}(\DB(Coh^{\mathcal{A}_n}(S \times \cdots \times S)))$ a module for some (twisted) Heisenberg algebras?
\end{quest}

In the last two questions, $e(\DB(Coh^{\mathcal{A}_n}(S \times \cdots \times S)$ is the Euler number (that is the alternating sum of Hochschild numbers) of $\DB(Coh^{\mathcal{A}_n}(S \times \cdots \times S))$ and $\HH^{\bullet}(\T)$ is the Hochschild cohomology of $\T$. If $\mathcal{A}_n$ is replaced by $\mathcal{S}_n$, the answer to Questions \ref{q3} and \ref{q4} is known to be ''yes''. Furthermore, it is known that hyper-Kähler geometry plays an important role in the proof of these results in the $ \mathcal{S}_n$ case \cite{gottsche}.
\end{section}

\begin{section}{Non-commutative relative compactified Prymian}
In this section, we study in details a categorical strongly crepant resolution of a singular compactified Prymian. This singular compactified Prymian first appeared in the work of Markushevich and Tikhomirov \cite{Marku-Tikho}. We recall briefly their construction in the first subsection.

\begin{subsection}{Markushevich-Tikhomirov's construction}

\bigskip

Let $X$ be a Del Pezzo surface of degree $2$ obtained as a double cover of $\mathbb{P}^2$ branched in a generic quartic curve $B_0$, $\mu : X \longrightarrow \mathbb{P}^2$ the double cover map, $B = \mu^{-1}(B_0)$ the ramification curve. Let $\Delta_0$ be a generic curve from the linear system $|-2K_X$, $\rho : S \longrightarrow X$ the double cover branched in $\Delta_0$ and $\Delta = \rho^{-1}(\Delta_0)$. Then $S$ is a K3 surface, and $H = \rho^*(-K_X)$ is a degree 4 ample divisor on $S$.

We will denote by $\iota$ (resp. $\tau$) the Galois involution of the double cover $\mu$ (resp. $\rho$). The plane quartic $B_0$ has $28$ bitangent lines $m_1,\cdots, m_{28}$ and $\mu^{-1}(m_i)$ is the union of two rational curves $l_i \cup l'_i$ meeting in 2 points. The 56 curves $l_i, l'_i$ are all the lines on $X$, that is, curves of degree $1$ with respect to $-K_X$. Further, the curves $C_i = \rho^{-1}(l_i)$, $C'_i = \rho^{-1}(l'_i)$ are conics on $S$, that is, curves of degree $2$ with respect to $H$. Each pair $C_i, C'_i$ meets in $4$ points, thus forming a reducible curve of arithmetic genus $3$ belonging to the linear system $|H|$.  We assume furthermore that $B_0$ and $\Delta_0$ are sufficiently generic. This implies that each line $l_i$ meets only one of the two lines $l_j, l'_j$ for $j \neq i$. The following is lemma $1.1$ of \cite{Marku-Tikho}:

\begin{lem}
The linear system $|H|$ is very ample and embeds $S$ as a quartic surface into $\mathbb{P}^3$. Every curve in $|H|$ is reduced and the only reducible members of $|H|$ are the $28$ curves $C_i+C'_i$ for $i=1, \dots, 28$.
\end{lem}

Let $m\geq 1$ and let $\mathcal{M}^{2m}$ be the moduli space of torsion sheaves with Mukai vector $(rk = 0, c_1 = H, \chi = 2m-2)$ which are semi-stable with respect to $H$. Markushevich and Tikhomirov prove the following (see proposition $1.2$, proposition $2.7$ and corollary $2.9$ of \cite{Marku-Tikho}):

\begin{prop}
The moduli space $\mathcal{M}^{2m}$ is a singular irreducible holomorphically symplectic variety of dimension $6$. It is singular in exactly $28$ points corresponding to the strictly semi-stable sheaves $\OO_{C_i}(m-1) \oplus \OO_{C'i}(m-1)$. Around each of these $28$ singular points, the moduli space $\mathcal{M}^{2m}$ is locally analytically equivalent to the contraction of the zero section of $\Omega_{\mathbb{P}^3} \rightarrow \mathbb{P}^3$.
\end{prop}

By varying the polarization, one gets symplectic desingularizations of $\mathcal{M}^{2m}$ which are deformation equivalent to $\mathrm{Hilb}^{3}(S)$. The idea of Markushevich and Tikhomirov is to study the fixed locus of a specific symplectic involution on $\mathcal{M}^{2m}$ in the hope it may provide a new hyper-Kähler manifold. Let $j$ be the involution of $\mathcal{M}^{2m}$ defined as:

\begin{equation*}
\begin{split}
j : &\, \, \mathcal{M}^{2m} \longrightarrow \mathcal{M}^{2m} \\
  &\, \,  \mathcal{F} \longrightarrow \mathcal{E}xt^{1}(\mathcal{F}, \OO_S((m-1)H)).
\end{split}
\end{equation*}
We consider the involution $\kappa = \tau \circ j$. The involution $\kappa$ is symplectic and its fixed locus is made of one four-dimensional components and $64$ zero-dimensional components. The four-dimensional component, denoted by $\mathcal{P}^{2m}$, is called the \textit{relative compactified Prymian} of $S$ by Tikhomirov and Markushevich and they prove the:

\begin{theo}[Theorem 3.4 \cite{Marku-Tikho}] The variety $\mathcal{P}^{2m}$ is a singular irreducible holomorphically symplectic variety of dimension $4$. It has exactly $28$ singular points corresponding to the sheaves $\OO_{C_i}(m-1) \oplus \OO_{C'_i}(m-1)$. Around each of these $28$ singular points, the Prymian $\mathcal{P}^{2m}$ is locally analytically equivalent to $\mathbb{C}^4/ \{-1,1\}$. The topological Euler number of $\mathcal{P}^{2m}$ is $268$.

\end{theo}
 Since $\mathcal{P}^{2m}$ is locally equivalent to $\mathbb{C}^4/ \{-1,1\}$ around its $28$ isolated singular points, it has no crepant resolution. As a consequence, it is not possible to construct a hyper-Kähler manifold starting from $\mathcal{P}^{2m}$. The singular variety $\mathcal{P}^{2m}$ is nevertheless studied in details by Markushevich and Tikhomirov and they prove, among other things, that it is birational to the quotient of $\mathrm{Hilb}^2(S)$ by a symplectic involution.
\end{subsection}

\begin{subsection}{Strongly crepant resolution of the relative compactified Prymian of Markushevich-Tikhomirov}

In this section, we construct a strongly crepant categorical resolution of $\mathcal{P}$. We study the Hochschild numbers and we show that they satisfy the Salamon's relation for Betti numbers of hyper-Kähler manifolds \cite{salamon}. We finally prove that this category can't be a deformation of the derived category of a projective variety, giving a counter-example to a conjecture of Kuznetsov (conjecture 5.8 of \cite{kuzCY}).

\begin{theo} \label{Prymian}
The singular variety $\mathcal{P}^{2m}$ admits a categorical strongly crepant resolution (denoted $\A_{\mathcal{P}^{2m}}$) which is a hyper-Kähler category of dimension $4$. The Hochschild cohomology numbers of $\A_{\mathcal{P}^{2m}}$ are:

\begin{itemize}
\item $\mathrm{hh}^{0} = \mathrm{hh}^8 = 1$,
\item $\mathrm{hh}^{2} = \mathrm{hh}^6 = 16$,
\item $\mathrm{hh}^{4} = 206$.
\end{itemize}
\end{theo}

\begin{proof}

Let $\widetilde{\mathcal{P}^{2m}}$ be the blow-up of $\mathcal{P}^{2m}$ along its $28$ singular points. We denote by  $E_1, \cdots, E_{28}$ are the exceptional divisors of the blow-up along the $28$. Example $7.1$ of \cite{kuz1} shows that there exists a semi-orthogonal decomposition:

\begin{equation*}
\DB(\widetilde{\mathcal{P}^{2m}}) = \left\langle \A_{\mathcal{P}^{2m}}, \OO_{E_1}, \OO_{E_1}(1),\cdots, \OO_{E_{28}}, \OO_{E_{28}}(1) \right\rangle,
\end{equation*}
where $\A_{\mathcal{P}^{2m}}$ is a categorical strongly crepant resolution of $\mathcal{P}^{2m}$. Markushevich and Tikhomirov prove that $\mathcal{P}^{2m}$ is a singular irreducible holmorphically symplectic variety, so that $\omega_{\mathcal{P}^{2m}} = \OO_{\mathcal{P}^{2m}}$ and $H^{\bullet}(\OO_{\mathcal{P}^{2m}}) = \mathbb{C}[t]/{t^3}$, with $t$ homogeneous of degree $2$. By Theorem \ref{hk-crepant}, we deduce that the category $\A_{\mathcal{P}^{2m}}$ is hyper-Kähler of dimension $4$ (with respect to its embedding in $\DB(\widetilde{\mathcal{P}^{2m}})$).

\bigskip

The Hodge numbers of $\widetilde{\mathcal{P}^{2m}}$ are computed by Grégoire Menet in the appendix and they are:

\begin{center}
\begin{tabular}{ccccccccc}
 & & & &$1$ & & & & \\
 & & &$0$ & & $0$& & & \\
 & &$1$ & &$42$ & &$1$ & & \\
 & $0$& &$0$ & &$0$ & &$0$ & \\
 $1$& &$14$ & &$176$ & &$14$ & &$1$ \\
\end{tabular}
\end{center}

By the Hochschild-Kostant-Rosenberg isomorphism, we have:

$$\mathrm{HH}_{k}(\DB(\widetilde{\mathcal{P}^{2m}})) = \bigoplus_{p-q = k} H^{p}(\widetilde{\mathcal{P}^{2m}}, \Omega^{q}_{\widetilde{\mathcal{P}^{2m}}}).$$

We deduce that the Hochschild homology numbers of $\DB(\widetilde{\mathcal{P}^{2m}})$ are:

\begin{itemize}
\item $\mathrm{hh}_0 = 262$
\item $\mathrm{hh}_{2} = \mathrm{hh}_{-2} = 16$
\item $\mathrm{hh}_{4} = \mathrm{hh}_{-4} = 1$,
\end{itemize}
and the others are zero. By corollary 7.5 of \cite{kuz3}, we have a graded direct sum decomposition:

\begin{equation*}
\mathrm{HH}_{\bullet}(\DB(\widetilde{\mathcal{P}^{2m}})) = \mathrm{HH}_{\bullet}(\A_{\mathcal{P}^{2m}}) \oplus \mathbb{C}^{56}.
\end{equation*}

We deduce that the Hochschild homology numbers of $\DB(\A_{\mathcal{P}^{2m}})$ are:
\begin{itemize}
\item $\mathrm{hh}_0 = 206$
\item $\mathrm{hh}_{2} = \mathrm{hh}_{-2} = 16$
\item $\mathrm{hh}_{4} = \mathrm{hh}_{-4} = 1$,
\end{itemize}
and the others are zero. But the Serre functor of the category $\A_{\mathcal{P}^{2m}}$ is the shift by $4$, hence we get an isomorphism of graded vector spaces $\mathrm{HH}_{\bullet}(\A_{\mathcal{P}^{2m}}) \simeq \mathrm{HH}^{\bullet-4}(\A_{\mathcal{P}^{2m}}) $. We then find that the Hochschild cohomology numbers of $\A_{\mathcal{P}^{2m}}$ are as stated.
\end{proof}

We notice that the Hochschild cohomology numbers of $\A_{\mathcal{P}^{2m}}$ satisfy the following relation:

\begin{equation*}
\sum_{j=1}^{4} (-1)^j (3j^2-2)\mathrm{hh}^{4-j} = \mathrm{hh}^4.
\end{equation*}
This relation is the four dimensional case of the Salamon relation for Betti numbers of hyper-Kähler manifolds \cite{salamon}. It is very tempting to believe that this relation holds for all hyper-Kähler categories.

\begin{conj} \label{salamonconj}
Let $\T$ be a full admissible subcategory of the derived category of a smooth projective variety. Assume that $\T$ is hyper-Käler of dimension $2r$. Then, we have the relation:

\begin{equation*}
\sum_{j=1}^{2r} (-1)^j (3j^2-r) .\mathrm{hh}^{2r-j} = \dfrac{r}{2}\mathrm{hh}^{2r}.
\end{equation*}
\end{conj}

Of course it would be interesting to first prove that this formula holds for the Hochschild numbers of the hyper-Kähler categories exhibited in Theorem \ref{nchilbert}. Note that the construction of relative compactified Prymians has been recently generalized to for arbitrary Enriques surfaces in \cite{arbasacca}. It is very likely that their construction will provide new examples of hyper-Kähler categories and that the Hochschild cohomology numbers of these categories will satisfy conjecture \ref{salamonconj}.

\bigskip

We close this section with a discussion of a conjecture made by Kuznetsov in \cite{kuzCY}:

\begin{conj}[\cite{kuzCY}, conjecture 5.8]
Let $X$ be a smooth projective variety of dimension $n$ and $\A \subset \DB(X)$ be a full admissible subcategory which is Calabi-Yau of dimension $n$. Then, there exists a birational morphism $X \longrightarrow X'$ such that $\A \simeq \DB(X')$.
\end{conj}

We will prove that this conjecture is far from being true:

\begin{prop} \label{defnonco}
The category $\A_{\mathcal{P}^{2m}}$ is not equivalent to the derived category of any projective variety. In fact, a small deformation of $\A_{\mathcal{P}^{2m}}$ is never equivalent to the derived category of a projective variety.
\end{prop}

The deformation theory we use here is the one developed in section 3 of this paper.

\begin{proof}

Let $\D$ be a deformation of $\A_{\mathcal{P}^{2m}}$. This is the data of a smooth connected algebraic variety $B$ and a smooth projective morphism $p : \X \longrightarrow B$ such that:
\begin{itemize}
\item $\X_{0} = \widetilde{\mathcal{P}^{2m}}$,

\item The category $\D$ is full admissible in $\DB(\X)$ and it is $B$-linear with the property that $ \EE_{0} := \EE \otimes_{\OO_{\X \times_{B} \X}} \OO_{\X_{0} \times \X_{0}} \in \DB(\X_{0} \times \X_{0})$ is the kernel of the projection $\DB(\X_{0}) \rightarrow \A_{\mathcal{P}^{2m}}$, where $\EE \in \DB(\X \times_{B} \X)$ is the kernel representing the projection functor $\DB(\X) \rightarrow \D$.

\end{itemize}

\bigskip

As in section $3$, for any $b \in B$, we denote by $\D_b$ the full admissible subcategory of $\DB(\X_b)$ whose projection functor is given by the kernel $\EE \otimes_{\OO_{\X \times_{B} \X}} \OO_{\X_{b} \times \X_{b}}$. Since $\OO_{\X_{0}} \in \A_{\mathcal{P}^{2m}}$ and $\CC(x_0) \in \A_{\mathcal{P}^{2m}}$, for generic $x_0 \in \X_0$, we know by lemmas \ref{defstructural} and \ref{defpoints} that there exists an open subset $0 \in U \subset B$ such that:

\begin{itemize}
\item for all $b \in U$, we have $\OO_{\X_b} \in \D_b$,

\item for all $b \in U$, we have $\CC(x_b) \in \D_b$, for generic $x_b \in \X_b$.

\end{itemize}

Furthermore, up to shrinking $U$, proposition \ref{smalldef} shows that for all $b \in U$, the category $\D_b$ is hyper-Kähler of dimension $4$ (with respect to its embedding in $\DB(\X_b)$).

\bigskip

Assume that there exists $b_0 \in B$ such that $\D_{b_0} \simeq \DB(Y)$, for some $Y$ projective. We immediately see that $Y$ is smooth projective of dimension $4$ with trivial canonical bundle. By hypothesis, the homological unit of $\D_{b_0}$ with respect to its embedding in $\DB(\X_{b_0})$ is $\mathbb{C}[t]/t^{3}$. Since $\OO_{\X_{b_0}} \in \D_{b_0}$, we have:

\begin{equation*}
H^{\bullet}(\OO_{\X_0}) \simeq \mathbb{C}[t]/{t^3},
\end{equation*}
with $t$ in degree $2$. By Theorem \ref{homounitiso}, we have $H^{\bullet}(\OO_Y) \simeq \CC[t]/{t^3}$, with $t$ in degree $2$. By proposition \ref{comtononcom}, we deduce that $Y$ is hyper-Kähler of dimension $4$. Theorem \ref{semiconti} and Theorem  \ref{Prymian} imply that the Hochschild numbers of $Y$ are:

\begin{itemize}
\item $\mathrm{hh}_0 = 206$
\item $\mathrm{hh}_{2} = \mathrm{hh}_{-2} = 16$
\item $\mathrm{hh}_{4} = \mathrm{hh}_{-4} = 1$.
\end{itemize}
But $Y$ being holomorphically symplectic, the Hochschild-Kostant-Rosenberg isomorphism implies that the Betti numbers of $Y$ are:

\begin{itemize}
\item ${b}_0 = b_8 = 1$
\item ${b}_{2} = b_{6} = 16$
\item $b_4 = 206$.
\end{itemize}
This is a contradiction. Indeed, it is proved in \cite{guan} that the second Betti number of a hyper-Kähler fourfold is either less than $8$ or equal to $23$. This concludes the proof that a small deformation of $\A_{\mathcal{P}^{2m}}$ can not be equivalent to the derived category of a projective variety.
\end{proof}

\bigskip

If one assumes that the deformation of $\A_{\mathcal{P}^{2m}}$ is Calabi-Yau of dimension $4$, contains $\OO_{\X}$ and the structure sheaf of a generic point, then one has a stronger statement than proposition \ref{defnonco}:

\begin{prop} \label{noncodeflong}
Let $\D$ be a deformation of $\A_{\mathcal{P}^{2m}}$ inside $\DB(\X)$ over $B$, for some $p : \X \longrightarrow B$ smooth projective. Assume that $\OO_{\X_b} \in \D_b$, that $\CC(x_b) \in\D_b$, for generic $x_b \in \X_b$ and that $\D_b$ is Calabi-Yau of dimension $4$, for all $b \in B$. Then, for all $b \in B$, the category $\D_b$ is never equivalent to the derived category of a projective variety.
\end{prop}

\begin{proof}
Assume that there exists $b_0 \in B$ such that $\D_b \simeq \DB(Y)$, where $Y$ is projective. By hypothesis, this immediately implies that $Y$ is smooth projective of dimension $4$ with trivial canonical bundle. By proposition \ref{smalldef}, the set of $b \in B$ such that $\D_b$ is hyper-Kähler of dimension $4$ (with respect to its embedding in $\DB(\X_b)$) is open (and non empty). 

\bigskip

Hence, up to shrinking $B$, one can assume that for all $b \neq b_0$, the category $\D_{b_0}$ is hyper-Kähler of dimension $4$ (with respect to its embedding in $\DB(\X_{b_0})$. Since $\D_{b_0}$ is Calabi-Yau of dimension $4$ and contains $\OO_{\X_{b_0}}$, we can apply proposition \ref{def-hk4} and we find that $\D_{b_0}$ is hyper-Kähler of dimension $4$ (with respect to its embedding in $\DB(\X_{b_0})$). One finishes the proof exactly as in the proof of proposition \ref{defnonco} above.

\end{proof}

\bigskip

 Of course, in view of conjecture \ref{def-HK}, one would expect that the hypotheses $\OO_{\X_b} \in \D_b$, $\CC(x_b) \in \D_b$ for generic $x_b \in \X_b$, and $\D_b$ CY-$4$, for all $b \in B$, are superfluous in the statement of proposition \ref{noncodeflong}. If this expectation is correct, then \textbf{any deformation} of $\A_{\mathcal{P}^{2m}}$ would \textbf{never} be equivalent to the derived category of a projective variety. In particular, this would imply that the moduli space of hyper-Kähler categories of dimension $4$ (if such an object exists) contains a component which is purely non-commutative!

\end{subsection}

\end{section}

\newpage

\bibliographystyle{alpha}

\bibliography{bibliHKC2}

\newpage

\appendix

\section{Hodge numbers of a resolution of the Markushevich--Tikhomirov relative compactified Prymian}

\begin{center}
 by {\Large Grégoire Menet} 

\end{center}

\bigskip
\bigskip
\bigskip

Let $\mathcal{P}^{0}$ be the Markushevitch--Tikhomirov variety defined in \cite{Markou}.
By Corollary 5.7 of \cite{Markou}, $\mathcal{P}^{0}$ is a projective irreducible symplectic V-manifold of dimension 4 whose singularities are 28 points of analytic type $(\mathbb{C}^{4}/\left\{\pm 1\right\},0)$. 
These singularities can be solved with one blow-up. Let $r:\widetilde{\mathcal{P}^{0}}\rightarrow\mathcal{P}^{0}$ be such a blow-up with
$D_{1},...,D_{28}$ the exceptional divisors.
We want to calculate the Hodge numbers of $\widetilde{\mathcal{P}^{0}}$.
We will proceed as follows; first, we will calculate the Hodge numbers of $\mathcal{P}^{0}$. Then, we will deduce the Hodge numbers of $\widetilde{\mathcal{P}^{0}}$ using the following lemma that can be found in \cite[Section 2.5]{Peters}.
\begin{lem}\label{PetersL}
The pullback $r^*:H^i(X,\mathcal{P}^{0})\rightarrow H^i(X,\widetilde{\mathcal{P}^{0}})$ is a morphism of Hodge structure.
\end{lem}
The reader non-familiar with the Hodge structure for a K\"ahler V-manifolds can also read
\cite[Section 1]{Fujiki}. 
Let's calculate the Hodge numbers of $\mathcal{P}^{0}$.

\begin{lem}\label{P0}
The Hodge diamond of $\mathcal{P}^{0}$ is:
$$\begin{matrix}
& & & &1& & & &\\
& & &0 & &0& & &\\
& &1 & & 14 & & 1 & &\\
& 0 & & 0 & & 0 & & 0 &\\
1 & & 14 & & 148 & & 14 & & 1.\\
& 0 & & 0 & & 0 & & 0 &\\
& &1 & & 14 & & 1 & &\\
& & &0 & &0& & &\\
& & & &1& & & & 
\end{matrix}$$
\end{lem}
\begin{proof}

By Paragraph 6.3 of \cite{Fujiki}, for a projective irreducible symplectic V-manifold of dimension four, we have:
$h^{0,0}=h^{2,0}=h^{4,0}=1$; $h^{3,0}=h^{1,0}=0$; $h^{1,1}=h^{3,1}$.
Hence, it only remains to calculate the three Hodge numbers $h^{1,1}$, $h^{2,1}$ and $h^{2,2}$. 

By Proposition 3.5 of \cite{Menet}, we have: $$b_{2}(\mathcal{P})=16,\ 
b_{3}(\mathcal{P}^{0})=0\ \text{and}\
b_{4}(\mathcal{P}^{0})=178.$$
Since we know that $h^{2,0}(\mathcal{P}^{0})=1$, we have $h^{1,1}(\mathcal{P}^{0})=14$.
Since $b_{3}(\mathcal{P}^{0})=0$, we have $h^{2,1}(\mathcal{P}^{0})=0$.
And since $h^{4,0}(\mathcal{P}^{0})=1$ and $h^{3,1}(\mathcal{P}^{0})=14$, we have $h^{2,2}(\mathcal{P}^{0})=148$.
\end{proof}

Now, we state the Hodge numbers of $\widetilde{\mathcal{P}^{0}}$.
\begin{prop}
The Hodge diamond of $\widetilde{\mathcal{P}^{0}}$ is:
$$\begin{matrix}
& & & &1& & & &\\
& & &0 & &0& & &\\
& &1 & & 42 & & 1 & &\\
& 0 & & 0 & & 0 & & 0 &\\
1 & & 14 & & 176 & & 14 & & 1.\\
& 0 & & 0 & & 0 & & 0 &\\
& &1 & & 42 & & 1 & &\\
& & &0 & &0& & &\\
& & & &1& & & & 
\end{matrix}$$
\end{prop}
\begin{proof}
Let $U=\mathcal{P}^{0}\smallsetminus \sing \mathcal{P}^{0}=\widetilde{\mathcal{P}^{0}}\smallsetminus \bigcup_{i=1}^{28} D_{i}$.
We consider the following exact sequence that we call $E_k$:
\begin{equation}
\xymatrix@C=10pt{ H^{k-1}(U,\C)\ar[r] & H^{k}(\widetilde{\mathcal{P}^{0}},U,\C)\ar[r]^g & H^{k}(\widetilde{\mathcal{P}^{0}},\C)\ar[r] & H^{k}(U,\C)\ar[r]& H^{k+1}(\widetilde{\mathcal{P}^{0}},U,\C).}
\label{exactSequence}
\end{equation}
We recall from Lemma 1.6 of \cite{Fujiki} that the restriction map:
\begin{equation}
H^{k}(\mathcal{P}^{0},\C)\simeq H^{k}(U,\C)
\label{Lemma1.6}
\end{equation}
is an isomorphism for all $0\leq k\leq6$.
Moreover, by Thom's isomorphism:
\begin{equation}
H^{k}(\widetilde{\mathcal{P}^{0}},U,\C)\simeq\bigoplus_{i=1}^{28}H^{k-2}(D_{i},\C),
\label{Thom}
\end{equation}
for all $k$. Furthermore, the map $g$ can be identified with the push-forward $j_*:\bigoplus_{i=1}^{28}H^{k-2}(D_{i},\C)\rightarrow H^{k}(\widetilde{\mathcal{P}^{0}},\C)$, where $j:\cup_{i=1}^{28}D_i\hookrightarrow \widetilde{\mathcal{P}^{0}}$ is the inclusion.

Now, we are ready to calculate the Hodge numbers.
First, the blow-up does not change the fundamental group, so $b_1(\widetilde{\mathcal{P}^{0}})=0$. Applying (\ref{Lemma1.6}) and (\ref{Thom}) to the exact sequence $E_2$, we obtain:
$$\xymatrix@C=10pt@R=10pt{ H^{1}(\mathcal{P}^{0},\C)=0 \ar@{=}[d]& \bigoplus_{i=1}^{28}H^{0}(D_{i},\C)\ar[rd]^{j_*}\ar@{=}[d] & & H^2(\mathcal{P}^{0},\C)\ar[ld]_{r^*}\ar@{=}[d] & \bigoplus_{i=1}^{28}H^{1}(D_{i},\C)=0\ar@{=}[d]\\
H^{1}(U,\C)\ar[r] & H^{2}(\widetilde{\mathcal{P}^{0}},U,\C)\ar[r]^g & H^{2}(\widetilde{\mathcal{P}^{0}},\C)\ar[r] & H^{2}(U,\C)\ar[r]& H^{3}(\widetilde{\mathcal{P}^{0}},U,\C).}$$
 So, it provides an isomorphism:
$$(r^*,j_*):H^2(\mathcal{P}^{0},\C)\oplus \left(\bigoplus_{i=1}^{28}H^{0}(D_{i},\C)\right)\rightarrow H^2(\widetilde{\mathcal{P}^{0}},\C).$$
We know that the cohomology classes of analytics subsets are of type (p,p) (see for example Theorem 11.31 of \cite{Voisin}). So by Lemma \ref{PetersL} and \ref{P0}:
$$h^{2,0}(\widetilde{\mathcal{P}^{0}})=h^{2,0}(\mathcal{P}^{0})=1\ \text{and}\ h^{1,1}(\widetilde{\mathcal{P}^{0}})=h^{1,1}(\mathcal{P}^{0})+28=42.$$

Applying the same method to the exact sequences $E_3$ and $E_4$, we obtain:
$$H^{3}(\widetilde{\mathcal{P}^{0}},\C)=0,$$
and the isomorphism:
$$(r^*,j_*):H^4(\mathcal{P}^{0},\C)\oplus \left(\bigoplus_{i=1}^{28}H^{2}(D_{i},\C)\right)\rightarrow H^4(\widetilde{\mathcal{P}^{0}},\C).$$
From  Theorem 11.31 of \cite{Voisin}, Lemma \ref{PetersL} and \ref{P0}, it follows:
$$h^{4,0}(\widetilde{\mathcal{P}^{0}})=h^{4,0}(\mathcal{P}^{0})=1,\ h^{3,1}(\widetilde{\mathcal{P}^{0}})=h^{3,1}(\mathcal{P}^{0})=14\ \text{and}\ h^{2,2}(\widetilde{\mathcal{P}^{0}})=h^{2,2}(\mathcal{P}^{0})+28=176.$$
It finishes the proof.
\end{proof}

\end{document}